\documentclass{amsart}

\usepackage{amssymb,mathrsfs,amscd,graphicx, array}
\allowdisplaybreaks
\usepackage[nocompress, noadjust]{cite} 
\usepackage{mathtools}

\usepackage[inline]{enumitem}

\usepackage{hyperref}
\usepackage{letltxmacro}
\LetLtxMacro{\oldsqrt}{\sqrt}
\renewcommand{\sqrt}[2][]{\,\oldsqrt[#1]{#2}\,}

\makeatletter
\def\@tocline#1#2#3#4#5#6#7{\relax
  \ifnum #1>\c@tocdepth 
  \else
    \par \addpenalty\@secpenalty\addvspace{#2}%
    \begingroup \hyphenpenalty\@M
    \@ifempty{#4}{%
      \@tempdima\csname r@tocindent\number#1\endcsname\relax
    }{%
      \@tempdima#4\relax
    }%
    \parindent\z@ \leftskip#3\relax \advance\leftskip\@tempdima\relax
    \rightskip\@pnumwidth plus4em \parfillskip-\@pnumwidth
    #5\leavevmode\hskip-\@tempdima
      \ifcase #1
       \or\or \hskip 1em \or \hskip 2em \else \hskip 3em \fi%
      #6\nobreak\relax
    \dotfill\hbox to\@pnumwidth{\@tocpagenum{#7}}\par
    \nobreak
    \endgroup
  \fi}
\makeatother

\usepackage{bbm}
\makeatletter

\def\greekbolds#1{%
 \@for\next:=#1\do{%
    \def\X##1;{%
     \expandafter\def\csname V##1\endcsname{\boldsymbol{\csname##1\endcsname}}
     }
   \expandafter\X\next;
  }
}

\greekbolds{alpha,beta,iota,gamma,lambda,nu,eta,Gamma,varsigma}

\def\make@bb#1{\expandafter\def
  \csname bb#1\endcsname{{\mathbb{#1}}}\ignorespaces}

\def\make@bbm#1{\expandafter\def
  \csname bb#1\endcsname{{\mathbbm{#1}}}\ignorespaces}

\def\make@bf#1{\expandafter\def\csname bf#1\endcsname{{\bf
      #1}}\ignorespaces} 

\def\make@gr#1{\expandafter\def
  \csname gr#1\endcsname{{\mathfrak{#1}}}\ignorespaces}

\def\make@scr#1{\expandafter\def
  \csname scr#1\endcsname{{\mathscr{#1}}}\ignorespaces}

\def\make@cal#1{\expandafter\def\csname cal#1\endcsname{{\mathcal
      #1}}\ignorespaces} 

\def\do@Letters#1{#1A #1B #1C #1D #1E #1F #1G #1H #1I #1J #1K #1L #1M
                 #1N #1O #1P #1Q #1R #1S #1T #1U #1V #1W #1X #1Y #1Z}
\def\do@letters#1{#1a #1b #1c #1d #1e #1f #1g #1h #1i #1j #1k #1l #1m
                 #1n #1o #1p #1q #1r #1s #1t #1u #1v #1w #1x #1y #1z}
\do@Letters\make@bb   \do@letters\make@bbm
\do@Letters\make@cal  
\do@Letters\make@scr 
\do@Letters\make@bf \do@letters\make@bf   
\do@Letters\make@gr   \do@letters\make@gr
\makeatother

\newcommand{\uln}{{\underline{n}}}

\newcommand{\abs}[1]{\lvert #1 \rvert}

\newcommand{\fl}[1]{\left\lfloor #1 \right\rfloor}

\newcommand{\wh}{\widehat}
\newcommand{\wt}{\widetilde}
\newcommand{\Lsymb}[2]{\genfrac{(}{)}{}{}{#1}{#2}}  

\newcommand{\brN}{\breve{\mathbb{N}}}
\newcommand{\qalg}[3]{\left(\frac{#1, #2}{#3}\right)}

\DeclareMathSymbol{\twoheadrightarrow} {\mathrel}{AMSa}{"10}

\DeclareMathOperator{\pr}{pr}

\DeclareMathOperator{\fchar}{char}

\DeclareMathOperator{\Ov}{Ov}
\DeclareMathOperator{\Nr}{Nr}

\DeclareMathOperator{\diag}{diag}

\DeclareMathOperator{\End}{End}
\DeclareMathOperator{\Hom}{Hom}

\DeclareMathOperator{\Gal}{Gal}
\DeclareMathOperator{\Mat}{Mat}

\DeclareMathOperator{\Tr}{Tr}
\DeclareMathOperator{\Nm}{N}  

\DeclareMathOperator{\GL}{GL}
\DeclareMathOperator{\SL}{SL}

\DeclareMathOperator{\Ssp}{SSp}

\newcommand{\Qbar}{\bar{\mathbb{Q}}}

\newcommand{\Z}{\mathbb Z}
\newcommand{\Q}{\mathbb Q}
\newcommand{\R}{\mathbb R}

\newcommand{\F}{\mathbb F}

\newcounter{thmcounter} 
\numberwithin{thmcounter}{section}  
\newtheorem{thm}[thmcounter]{Theorem}
\newtheorem{lem}[thmcounter]{Lemma}

\newtheorem{prop}[thmcounter]{Proposition}
\theoremstyle{definition}
\newtheorem{defn}[thmcounter]{Definition}

\newtheorem{rem}[thmcounter]{Remark}

\numberwithin{equation}{section}

\newtheoremstyle{notitle}  
  {}
  {}
  {\itshape}
  {}
  {}
  {\ }
  {.5em}
  {}
\theoremstyle{notitle}

 \title[Superspecial abelian surfaces]{On Superspecial abelian surfaces over finite
 fields III}
\author{Jiangwei Xue}

\address{(Xue) Collaborative Innovation Center of Mathematics, School of
  Mathematics and Statistics, Wuhan University, Luojiashan, 430072,
  Wuhan, Hubei, P.R. China}   

\address{(Xue) Hubei Key Laboratory of Computational Science (Wuhan
  University), Wuhan, Hubei,  430072, P.R. China.}
\email{xue\_j@whu.edu.cn}

\author{Chia-Fu Yu}

\address{(Yu) Institute of Mathematics,
  Academia Sinica and NCTS, Astronomy-Mathematics
  Building, No. 1, Sec. 4, Roosevelt Road, Taipei 10617, TAIWAN.}

\email{chiafu@math.sinica.edu.tw}

\author{Yuqiang Zheng}

\address{(Zheng) School of Mathematics and Statistics, 
  Wuhan University, Luojiashan, 430072, Wuhan, Hubei, P.R. China}
\email{zhhhhxhq@whu.edu.cn}

\curraddr{(Zheng) Academy of Mathematics and Systems Science, Chinese
  Academy of Science, No. 55, Zhongguancun East Road, Beijing 100190,
  China}

\email{zhengyq@amss.ac.cn}

\begin{document}
\date{\today} 
 \subjclass[2020]{11R52, 11G10} 
 \keywords{superspecial  abelian 
 surfaces, quaternion algebra, Bass order, conjugacy classes of
 arithmetic subgroups}
\begin{abstract}
  In the paper [On superspecial abelian surfaces over finite fields {II}.
  {\em J. Math. Soc. Japan}, 72(1):303--331, 2020], Tse-Chung Yang and the
  first two current authors computed explicitly the number $\lvert \mathrm{SSp}_2(\mathbb{F}_q)\rvert$ of isomorphism classes
  of superspecial abelian surfaces over an arbitrary finite field $\mathbb{F}_q$
   of
  \emph{even} degree over the prime field $\mathbb{F}_p$. There it was
  assumed that certain commutative $\mathbb{Z}_p$-orders satisfy an \'etale condition
  that excludes the primes $p=2, 3, 5$. We treat these
  remaining primes in the present paper, where the computations are
  more involved because of the ramifications. This 
completes the calculation of $\lvert \mathrm{SSp}_2(\mathbb{F}_q)\rvert$ in the even degree
case. The odd degree case was previous treated by Tse-Chung Yang and
the first two current authors in [On superspecial abelian surfaces over finite fields.{\em Doc. Math.}, 21:1607--1643, 2016]. Along the proof of our main theorem, we give the classification of lattices over local quaternion Bass orders,
which is a new input to our previous works. 
\end{abstract}

\maketitle


\section{Introduction}

Throughout this paper, $p$ denotes a prime number, $q=p^a$ a power of
$p$, and $\F_q$ the finite field of $q$-elements. We reserve $\bbN$
for the set of strictly positive integers.  Let $k$ be a field of
characteristic $p$, and $\bar{k}$  an algebraic closure of $k$.  An
abelian variety over $k$ is said to be {\it supersingular} if it is
isogenous to a product of supersingular elliptic curves over $\bar k$;
it is said to be {\it superspecial} if it is isomorphic to a product
of supersingular elliptic curves over $\bar k$.  For any $d\in \bbN$,
denote by $\Ssp_d(\F_q)$ the set of $\F_q$-isomorphism classes of
$d$-dimensional superspecial abelian varieties over $\F_q$.  The
classification of supersingular elliptic curves (namely, the $d=1$ case) over finite fields
were carried out by Deuring \cite{Deuring1950, Deuring-1941}, Eichler
\cite{eichler-CNF-1938}, Igusa\cite{igusa}, Waterhouse \cite{waterhouse:thesis}  and many others since the
1930s.

In a series of papers
\cite{xue-yang-yu:sp_as,xue-yang-yu:sp_as2,xue-yang-yu:ECNF,
  xue-yang-yu:num_inv}, Tse-Chung Yang and the first two current authors attempt to calculate the cardinality
$\abs{\Ssp_d(\F_q)}$ explicitly in the case $d=2$. More precisely, it
is shown in \cite{xue-yang-yu:sp_as} that for every fixed $d>1$,
$\abs{\Ssp_d(\F_q)}$ depends only on the parity of the degree
$a=[\F_q:\F_p]$, and an explicit formula of $\abs{\Ssp_2(\F_q)}$ is
provided for the odd degree case.  The most involving part of
this explicit calculation is carried out prior in \cite{xue-yang-yu:ECNF,
  xue-yang-yu:num_inv}, which counts the number of isomorphism classes
of abelian surfaces over $\F_p$ within the simple isogeny class
corresponding to the Weil $p$-numbers
$\pm\sqrt{p}$. For the even degree case, an explicit formula of
$\abs{\Ssp_2(\F_q)}$ is obtained in \cite{xue-yang-yu:sp_as2} under \emph{a
mild condition on $p$} (see Remark~3.7 of loc. cit.), which holds for all $p\geq 7$. We treat the
remaining primes $p\in \{2, 3, 5\}$ in the present paper, thus
completing the calculation of $\abs{\Ssp_2(\F_q)}$ in the even degree
case. 

For the rest of the paper, we assume that $q=p^a$ is an \emph{even} power of $p$.
All isogenies and isomorphisms are over the base field $\F_q$ unless
specified otherwise.  The set $\Ssp_2(\F_q)$ naturally partitions into
subsets by isogeny equivalence, which can be parametrized by
(multiple) Weil numbers (see \cite[\S 4.1]{xue-yang-yu:sp_as}). For
each integer $n\in \bbN$, let $\zeta_n$ be a primitive $n$-th root of unity,
and $\pi_n$ be the Weil $q$-number $(-p)^{a/2}\zeta_n$. By the
Honda-Tate theorem, there is a unique simple abelian variety $X_n/\F_q$ up to isogeny corresponding to the
$\Gal(\Qbar/\Q)$-conjugacy class of $\pi_n$.  Moreover, the $X_n$'s
are mutually non-isogenous for distinct $n$.  Thanks to the Manin-Oort
Theorem \cite[Theorem~2.9]{Yu-End-QM-2013}, a simple abelian variety over $\F_q$ is supersingular if and
only if it is isogenous to $X_n$ for some $n$. Let $d(n)$ be the
dimension of $X_n$. The formula for $d(n)$ is given in \cite[\S
3]{xue-yang-yu:sp_as}.  When $p\in \{2,3,5\}$, we have
\begin{itemize}
\item $d(n)=1$ if and only if $n\in\{1,2, 3, 6\}$ or $(n, p)\in \{(4,
  2), (4, 3)\}$;
\item $d(n)=2$ if and only if $(n,p)=(4,5)$ or $n\in \{5,8,10,12\}$. 
\end{itemize}

Given a superspecial abelian surface $X/\F_q$, we have two
cases to consider:
\begin{enumerate}[label=(\Roman*)]
\item \textbf{the isotypic case} where $X$ is isogenous to 
  $X_n^{2/d(n)}$ for some $n\in \bbN$ with $d(n)\leq 2$; 
\item \textbf{the non-isotypic case} where $X$ is isogenous to $X_\uln:=X_{n_1}\times X_{n_2}$ for a pair $\uln=(n_1,
n_2)\in \bbN^2$ with $n_1<n_2$ and
$d(n_1)=d(n_2)=1$. 
\end{enumerate}
Let $o(n)$ (resp.~$o(\uln)$) denote the number of isomorphism classes
of superspecial abelian surfaces over $\F_q$ that are isogenous to
$X_n^{2/d(n)}$ (resp. $X_\uln$). It was shown in \cite[\S
3.2]{xue-yang-yu:sp_as2} that 
\begin{gather}
  o(1)=o(2)=1,\ o(3)=o(6), \ o(5)=o(10); \label{eq:2} \\
  o(1,3)=o(2,6), \  o(1,4)=o(2,4), \ o(1,6)=o(2,3), \
o(3,4)=o(4,6). \label{eq:3}
\end{gather}
Thus we have 
\begin{equation}
  \label{eq:1}
  \begin{split}
\abs{\Ssp_2(\F_q)}=&2+2o(3)+o(4)+2o(5)+o(8)+o(12) \\
          &+o(1,2)+2o(2,3)+2o(2,4)+2o(2,6)+2o(3,4)+o(3,6).
  \end{split}
\end{equation}

As mentioned before, the value of each $o(n)$ or $o(\uln)$ in (\ref{eq:1})
has been worked out in \cite{xue-yang-yu:sp_as2} \emph{conditionally} on $p$.
To make explicit this condition, we uniformize the notation. For each
$r\in \bbN$, let us denote
\[\brN^r:=\{\uln=(n_1, \cdots, n_r)\in \bbN^r\mid 0<n_1<\cdots
<n_r\}. \]
In particular, if $r=1$, then $\brN^r=\bbN$ and we drop the underline
from $\uln$. For each  $\uln\in \brN^r$ with $r$ arbitrary, we define 
\begin{equation}
  \label{eq:5}
   A_\uln:=\frac{\Z[T]}{\left(\prod_{i=1}^r \Phi_{n_i}(T)\right)}, \qquad
K_\uln:=\frac{\Q[T]}{\left(\prod_{i=1}^r \Phi_{n_i}(T)\right)}\simeq \prod_{i=1}^r
\Q(\zeta_{n_i}), 
\end{equation}
where $\Phi_n(T)\in \Z[T]$ is the $n$-th cyclotomic
polynomial. Clearly, $A_\uln$ is a $\Z$-order in $K_\uln$, so it is contained in
the unique  maximal order $O_{K_\uln}:=\prod_{i=1}^r \Z[T]/(\Phi_{n_i}(T))$. 
Let $\uln\in \brN^r$ with $r\in \{1,2\}$ be an $r$-tuple appearing
in (\ref{eq:1}).  In the proof of
\cite[Theorem~3.3]{xue-yang-yu:sp_as2}, it is assumed that 
\begin{equation}
  \label{eq:6}
A_{\uln, p}:=A_\uln\otimes \Z_p \quad \text{is an \'etale $\Z_p$-order}.   
\end{equation}
This condition fails precisely in the following two situations:
\begin{enumerate}[label=(C\arabic*)]
\item $p$ is  ramified in $\Q(\zeta_{n_i})$ for some $1\leq i\leq r$, or  
\item $p$ divides the index $[O_{K_\uln}: A_\uln]$. 
\end{enumerate}
If $r=1$, then $A_n$ coincides with $O_{K_n}$, so (C2) is possible
only if $r=2$. For the reader's convenience, we list the  indices
$i(\uln):=[O_{K_\uln}: A_\uln]$ from \cite[Table~1]{xue-yang-yu:sp_as2}:

\begin{center}
  \renewcommand{\arraystretch}{1.2}
\begin{tabular}{*{7}{|>{$}c<{$}}|}
\hline
\uln & (1, 2) & (2, 3) & (2, 4) & (2, 6) & (3, 4) & (3, 6)\\
\hline
i(\uln)& 2 & 1 & 2 & 3 & 1 & 4\\
\hline
\end{tabular}
\end{center}

  


\begin{thm}\label{thm:main}
Let $\uln\in \brN^r$  be an $r$-tuple appearing
in (\ref{eq:1}), and $p$ be a prime satisfying (C1) or
(C2). Then the values of $o(\uln)$ for each $p$ are given by the following table
\begin{center}
\begin{tabular}{*{13}{|>{$}c<{$}}|}
\hline
  \uln & 3 & 4 & 5 & 8 & 12 & (1, 2) & (2, 3) & (2, 4) & (2, 6) & (3,
                                                                  4) &
                                                                       \multicolumn{2}{c|}{$(3,
                                                                       6)$}
  \\
\hline
p & 3 & 2 & 5 & 2 & 2,3 & 2 & 3 &2 & 3 & 2, 3& 2 & 3\\
\hline
o(\uln)& 2 & 2 & 1 & 1 & 3 & 3 & 1 & 2 & 3 & 2 & 8 & 2\\
\hline
\end{tabular}
\end{center}
Moreover, the number of isomorphism classes of superspecial abelian surfaces over a finite field $\F_q$ of even degree over $\F_p$ with $p\in \{2, 3, 5\}$ is given by 
\begin{equation}\label{eq:9}
    \abs{\Ssp_2(\F_q)}=\begin{cases}
    49  &\text{if } p=2,\\
    45 &\text{if } p=3,\\
    47 &\text{if } p=5.\\
    \end{cases}
\end{equation}
\end{thm}

\begin{rem}\label{rem:arith}
  We provide an arithmetic interpretation of the values $o(\uln)$. Let
  $D=D_{p, \infty}$ be the unique quaternion $\Q$-algebra up to
  isomorphism ramified precisely at $p$ and $\infty$, and $\Mat_2(D)$
  be the algebra of $2\times 2$ matrices over $D$. Fix a maximal
  $\Z$-order $\calO$ in $D$. As explained in \cite[\S1,
  p.~304]{xue-yang-yu:sp_as2}, up to isomorphism, the arithmetic group
  $\GL_2(\calO)$  depends only on $p$ and not on the choice of $\calO$.
  An
  element $x\in \GL_2(\calO)$ of finite group order\footnote{Unfortunately,
    the word ``order'' plays double duties in this paragraph: 
    for the order inside an algebra and also for the order of a group
    element.  To make a distinction, we always insert the word ``group'' when the second
    meaning  applies.} 
 is semisimple, so its
  minimal polynomial over $\Q$ in $\Mat_2(D)$ is of the form
  $P_{\uln}(T):=\prod_{i=1}^r \Phi_{n_i}(T)$ for some
  $\uln=(n_1, \cdots, n_r)\in \brN^r$. It is not hard to show that
  $r\leq 2$ (see \cite[\S3.1]{xue-yang-yu:sp_as2}).  A Galois
  cohomological argument shows that $o(\uln)$ counts the number of
  conjugacy classes of elements of $\GL_2(\calO)$ with minimal
  polynomial $P_{\uln}(T)$, and $\abs{\Ssp_2(\F_q)}$ is equal to the
  total number
  of conjugacy classes of elements of finite group order in
  $\GL_2(\calO)$ (see
  \cite[Proposition~1.1]{xue-yang-yu:sp_as2}). Actually, this arithmetic interpretation works for $\GL_d(\calO)$ with any $d\ge 2$, not just for $d=2$.
\end{rem}

The proof of Theorem~\ref{thm:main} will occupy the remaining part of
the paper. In Section~\ref{sec:gen-strategy}, we recall from \cite[\S3.1]{xue-yang-yu:sp_as2} the
general strategy for computing $o(\uln)$. The isotypic case
(i.e.~$r=1$) will be treated in Section~\ref{sec:isotypic-case}, and
the non-isotypic case (i.e.~$r=2$) will be treated in
Section~\ref{sec:non-isotypic-case}.

\section{General strategy for computing $o(\uln)$}
\label{sec:gen-strategy}
Keep the notation and the assumptions of the previous section. We
recall from \cite[\S3.1]{xue-yang-yu:sp_as2} and \cite[\S6.4]{xue-yang-yu:sp_as} the general strategy for calculating $o(\uln)$ with
$\uln\in \brN^r$ for $r\leq 2$. 
  Based on the arithmetic interpretation of $o(\uln)$ in Remark~\ref{rem:arith}, we further
 provide  a lattice description of $o(\uln)$. Indeed, it is via
 this lattice description that the value of each $o(\uln)$ is calculated.


 Let $V=D^2$ be the unique simple left $\Mat_2(D)$-module, which is at
 the same time a
 right $D$-vector space of dimension $2$.  Let  $M_0:=\calO^2$ be the
 standard right $\calO$-lattice in $V$, whose endomorphism ring
 $\End_{\calO}(M_0)$ is just $\Mat_2(\calO)$. For each element $x\in \GL_2(\calO)$ of finite group order with
 minimal polynomial $P_{\uln}(T)\in \Z[T]$, there is a canonical
 embedding $A_\uln=\Z[T]/(P_\uln(T))\hookrightarrow \Mat_2(\calO)$
 sending $T$ to $x$. This embedding equips $M_0$ with an
 $(A_\uln, \calO)$-bimodule structure, or equivalently, a faithful
left $A_\uln\otimes_\Z\calO^{\mathrm{opp}}$-module structure.
Similarly, $V$ is equipped with a faithful left $K_\uln\otimes_\Q
D^{\mathrm{opp}}$-module structure. The
 canonical involution induces an isomorphism between the opposite
ring $\calO^{\mathrm{opp}}$ and $\calO$ itself (and similarly between
 $D^{\mathrm{opp}}$ and $D$), so we put
 \begin{equation}
  \label{eq:7}
\scrA_\uln:=A_\uln\otimes_\Z\calO,  \quad \text{and} \quad\scrK_\uln:=K_\uln\otimes_\Q D. 
\end{equation}
Clearly, $\scrA_\uln$ is a $\Z$-order in the semisimple $\Q$-algebra
$\scrK_\uln$. It has been shown in \cite[p.~309]{xue-yang-yu:sp_as2}
that the $\scrK_\uln$-module structure on $V$ is uniquely determined
by the $r$-tuple $\uln$.  From
\cite[Theorem~6.11]{xue-yang-yu:sp_as} (see also
\cite[Lemma~3.1]{xue-yang-yu:sp_as2}),  the above construction induces
a bijection between the following two finite sets:
\begin{equation}
  \label{eq:10}
  \left\{\parbox{1.7in}{conjugacy classes of elements of $\GL_2(\calO)$ with
      minimal polynomial $P_\uln(x)$}\right\}\longleftrightarrow
  \left\{\parbox{1.2in}{isomorphism classes of $\scrA_\uln$-lattices in
      the left $\scrK_\uln$-module $V$}\right\}
\end{equation}
Therefore, we have $o(\uln)=\abs{\scrL(\uln)}$,  where $\scrL(\uln)$
denote the set on the right. 

Now fix a pair $(\uln, p)$ in Theorem~\ref{thm:main} and
in turn a left $\scrK_\uln$-module $V$.  Given an $\scrA_\uln$-lattices
$\Lambda\subset V$, we write $[\Lambda]$ for its isomorphism class,
and $O_\Lambda$ for its endomorphism ring
$\End_{\scrA_\uln}(\Lambda)\subset \End_{\scrK_\uln}(V)$.  As a
convention, the endomorphism algebra $\scrE_\uln:=\End_{\scrK_\uln}(V)$ acts on
$V$ from the left, so it coincides with the centralizer of $K_\uln$ in
$\Mat_2(D)$. Two $\scrA_\uln$-lattices $\Lambda_1$ and $\Lambda_2$ in
$V$ are isomorphic if and only if there exists
$g\in \scrE_\uln^\times$ such that
$\Lambda_1=g\Lambda_2$.  

For each prime $\ell\in \bbN$, we use the subscript $_\ell$ to indicate
$\ell$-adic completion. For example,  $\scrA_{\uln, \ell}$ (the
$\ell$-adic completion of $\scrA_\uln$) is a
$\Z_\ell$-order in the semisimple $\Q_\ell$-algebra $\scrK_{\uln,
  \ell}$, and $\Lambda_\ell$ is an $\scrA_{\uln,
  \ell}$-lattice in $V_\ell$. 
 For each prime $\ell$, let $\scrL_\ell(\uln)$ denote the set of
isomorphism classes of $\scrA_{\uln, \ell}$-lattices in the left
$\scrK_{\uln, \ell}$-module $V_\ell$.  For almost all primes $\ell$,
the $\Z_\ell$-order $\scrA_{\uln, \ell}$ is 
maximal  in $\scrK_{\uln,
  \ell}$, in which case both of the following hold by
\cite[Theorem~26.24]{curtis-reiner:1}: 
\begin{enumerate}
\item[(i)] $\Lambda_\ell$ is uniquely determined up to
  isomorphism (i.e.~$\abs{\scrL_\ell(\uln)}=1$), and
\item[(ii)]   $O_{\Lambda, \ell}$ is maximal in $\scrE_{\uln,
  \ell}$. 
\end{enumerate}
Let $S(\uln, p)$ be the finite set of primes $\ell$ for which
$\scrA_{\uln, \ell}$ is non-maximal. 
The profinite completion $\Lambda\mapsto
\wh\Lambda:=\prod_{\ell}\Lambda_\ell$ induces a surjective map
\begin{equation}
  \Psi: \scrL(\uln)\to \prod_\ell\scrL_\ell(\uln)\simeq
  \prod_{\ell\in S(\uln, p)}\scrL_\ell(\uln). 
\end{equation}
Two
$\scrA_\uln$-lattices $\Lambda_1$ and $\Lambda_2$ in
$V$ are said to be in the \emph{same genus} if 
$\Psi([\Lambda_1])=\Psi([\Lambda_2])$, or equivalently,
$(\Lambda_1)_\ell\simeq (\Lambda_2)_\ell$ for every prime $\ell$. The
fibers of $\Psi$ partition $\scrL(\uln)$ into a disjoint union of
genera. Let $\scrG(\Lambda):=\Psi^{-1}(\Psi([\Lambda]))\subseteq
\scrL(\uln)$  be the fiber of $\Psi$ over $\Psi([\Lambda])$,
that is, the set of isomorphism classes of $\scrA_\uln$-lattices
in  the genus of $\Lambda$.  
From \cite[Proposition~1.4]{Swan-1988},
we have
\begin{equation}
  \label{eq:12}
  \abs{\scrG(\Lambda)}=h(O_\Lambda), 
\end{equation}
where $h(O_\Lambda)$ denote the class number of
$O_\Lambda$. In other words, $h(O_\Lambda)$ is  the number of locally principal right (or
equivalently, left) ideal
classes of $O_\Lambda$. 

Therefore, the computation of $o(\uln)$ can be carried out in the
following two steps:
\begin{enumerate}[label=(Step~\arabic*)]
  \item Classify the genera of $\scrA_{\uln}$-lattices in the left
        $\scrK_{\uln}$-module $V$. Equivalently, classify the
        isomorphism classes of $\scrA_{\uln, \ell}$-lattices in
        $V_\ell$ for each  $\ell\in S(\uln, p)$.
      \item Pick a lattice $\Lambda$ in each genus and write down its
        endomorphism ring $O_\Lambda$ (at least locally at each prime
        $\ell$). The number $o(\uln)$ is obtained by summing up the
        class numbers $h(O_\Lambda)$ over all genera. 
      \end{enumerate}

      \begin{rem}
      The reason that condition (\ref{eq:6}) is assumed throughout the
calculations in \cite{xue-yang-yu:sp_as2} is to  make sure that the $\Z_p$-order $\scrA_{\uln, p}$ is a product of Eichler
orders \cite[Remark~3.7]{xue-yang-yu:sp_as2}. In our 
setting, $p$ satisfies condition (C1) or (C2),  so  $\scrA_{\uln, p}$
becomes more complicated. This is precisely why the primes $p\in \{2, 3,
5\}$ are treated separately from the rest of the primes.   Luckily for
us, many $\scrA_{\uln, p}$ turn out to be Bass orders (see
Definition~\ref{defn:Bass} below), which makes the classification of
$\scrA_{\uln, p}$-lattices more  manageable.
     \end{rem}
%

      \section{The isotypic case}
      \label{sec:isotypic-case}
In this section, we calculate the values of $o(n)$ for $n \in \{3,4,5,8,12\}$
and $p | n$.  Keep the notation of previous sections. In particular,
$D=D_{p, \infty}$ is the unique quaternion $\Q$-algebra ramified
precisely at $p$ and $\infty$, and $\calO$ is a maximal $\Z$-order in
$D$.  
Since $A_n$ is the maximal order in the $n$-th cyclotomic field $K_n$, and $\calO$ has reduced
discriminant $p$, we have $S(n, p)=\{p\}$. In other words, the
$\ell$-adic completion $\scrA_{n, \ell}$ is non-maximal in
$\scrK_{n, \ell}$ if and only if $\ell=p$. It turns out that
$\scrA_{n, p}$ is always a Bass order in the quaternion $K_{n,
  p}$-algebra $\scrK_{n, p}$. Therefore, the classification of genera of the lattice set $\scrL (n)$ is then  reduced
  to the classification of lattices over local quaternion Bass orders. 
  

\subsection{Classification of lattices over local quaternion Bass orders}
\label{subsec:lattices-over-local}
The main references for this section are \cite{Brzezinski-1983, Brzezinski-crelle-1990} and
\cite[\S37]{curtis-reiner:1}.  Let $F$ be  a
nonarchimedean local field, and
$O_F$ be its ring of integers.  Fix a uniformizer $\varpi$ of $F$ and
denote the the finite residue field $O_F/\varpi O_F$ by $\grk$. Let $B$ be a
finite dimensional separable $F$-algebra \cite[Definition~7.1 and
Corollary~7.6]{curtis-reiner:1},  and $\scrO$ be an $O_F$-order (of
full rank) in
$B$.  We write  $\Ov(\scrO)$ for the finite set of \emph{overorders} of
$\scrO$, i.e.~$O_F$-orders in $B$ containing
$\scrO$.  A \emph{minimal overorder} of $\scrO$ is a minimal member of
$\Ov(\scrO)\smallsetminus \{\scrO\}$ with respect to inclusion.  

\begin{defn}\label{defn:Bass}
  An $O_F$-order $\scrO$ in $B$ is \emph{Gorenstein} if its dual lattice
  $\scrO^\vee:=\Hom_{O_F}(\scrO, O_F)$ is 
  projective as a left (or right) $\scrO$-module. It is
  called a \emph{Bass order} if every member of $\Ov(\scrO)$  is
  Gorenstein. It is called a \emph{hereditary order} if  
 every  left ideal of $\scrO$ is projective as a left $\scrO$-module. If $\scrO$ is the intersection of two maximal orders, then it is called an \emph{Eichler order}.
\end{defn}

%

We have the following inclusions of orders:
\[ \text{(maximal)}\subset \text{(herediary)} \subset \text{(Eichler)}\subset \text{(Bass)}\subset \text{(Gorenstein)}.\]
If $B$ is division, then Eichler orders are also maximal. 
Bass notes in \cite{Bass-MathZ-1963} that Gorenstein orders are ubiquitous. 

Let $I$ be a fractional left $\scrO$-ideal (of full rank) in $B$. We
say $I$ is \emph{proper  over $\scrO$} if its associated left order
$O_l(I):=\{x\in B\mid xI \subseteq I\}$ coincides with $\scrO$.  From
\cite[Example~2.6 and Corollary~2.7]{Brzezinski-loc-Princ}, the
following lemma provides an
equivalent characterization of Gorenstein orders in certain types of $F$-algebras:

\begin{lem}\label{lem:Gorenstein}
Suppose that $B$ is either a commutative algebra or a quaternion
$F$-algebra. Then   $\scrO$ is Gorenstein if and only if every proper fractional
left $\scrO$-ideal $I\subset B$ is principal (i.e.~there exists $x\in
B^\times$ such that $I=\scrO x$). 
\end{lem}
In the quaternion case, the above lemma can also be obtained by
combining \cite[Condition~G4 or G4', p.~1364]{Drozd-Kirichenko-Roiter-1967} and
\cite[Theorem~1]{Kaplansky-quat-invertible}. The lemma no longer holds in general
for orders in more complicated algebras.  If every proper fractional
left $\scrO$-ideal $I$ is principal, then $\scrO$ is
Gorenstein, but the converse is not necessarily true. 
See
\cite[p.~220]{Kaplansky-quat-invertible} and 
\cite[p.~535]{Brzezinski-loc-Princ} for some examples. Nevertheless, a
proper fractional
left ideal over a Gorenstein order  is always
left projective according to
\cite[Theorem~5.3, pp.~253--255]{Roggenkamp-Latt-II}. However, unlike the situation over commutative rings, a projective module over a non-commutative ring may not be locally free.  
Brzezinski \cite[Proposition~2.3]{Brzezinski-loc-Princ} gave a precise characterization of the orders $\scrO$ such that every proper fractional
left $\scrO$-ideal $I\subset B$ is principal (Such orders are called \emph{strongly Gorenstein} by him).

For the rest of Section~\ref{subsec:lattices-over-local}, we assume that $\fchar(F)\neq 2$ and
$B$ is a quaternion $F$-algebra.   The reduced trace and reduced
norm maps of $B$ are denoted by $\Tr: B\to F$ and $\Nr: B\to F$
respectively.  We write $\grd(\scrO)$ for the reduced discriminant of
$\scrO$, which is a nonzero integral ideal of $O_F$. 
From \cite[Proposition~1.2]{Brzezinski-1983}, $\scrO$ is
hereditary if and only if $\grd(\scrO)$ is square-free. Thus if $B$ is
division, then $\scrO$ is hereditary if and only if $\scrO$ is the
unique maximal order of $B$; if $B\simeq \Mat_2(F)$, then $\scrO$ is
hereditary if and only if $\scrO$ is isomorphic to $\Mat_2(O_F)$ or  $\left[\begin{smallmatrix}
  O_F & O_F\\ \varpi O_F & O_F
\end{smallmatrix}\right]$. 


\begin{thm}
  The following are equivalent:
  \begin{enumerate}
  \item[(a)] every left $\scrO$-ideal is generated by at most 2
    elements;
  \item[(b)] $\scrO$ is Bass;
    \item[(c)] every indecomposable $\scrO$-lattice is isomorphic to
      an ideal of $\scrO$;
   \item[(d)] $\scrO\supseteq O_L$ for some 
     semisimple quadratic $F$-subalgebra   $L\subseteq B$.
  \end{enumerate}
\end{thm}
Indeed, the implications $(a)\Rightarrow (b) \Rightarrow (c)$ hold in
much more general settings according to 
\cite[\S37]{curtis-reiner:1}. The implication $(c) \Rightarrow (a)$ is
proved by Drozd, Kirichenko and Roiter
\cite{Drozd-Kirichenko-Roiter-1967} (see
\cite[p.~790]{curtis-reiner:1}). Lastly, the equivalence
$(b)\Leftrightarrow (d)$ is proved by Brzezinski
\cite[Proposition~1.12]{Brzezinski-crelle-1990}. See Chari
et al.~\cite{Voight-basic-orders} for more
characterization of quaternion Bass orders. 

We recall the notion of \emph{Eichler invariant} following
\cite[Definition~1.8]{Brzezinski-1983}. 
\begin{defn}
  Let   $\grk'/\grk$ be the unique
quadratic field extension.  When $\scrO\not\simeq
\Mat_2(O_F)$, the quotient of $\scrO$ by its Jacobson radical
$\grJ(\scrO)$ falls into the following three cases: 
\[\scrO/\grJ(\scrO)\simeq \grk\times \grk, \qquad \grk,
\quad\text{or}\quad \grk', \]
and the \emph{Eichler invariant} $e(\scrO)$  is defined to be
$1, 0, -1$ accordingly.  As a convention, if $\scrO\simeq
\Mat_2(O_F)$, then its Eichler invariant is defined to be
$2$.
\end{defn}


For example, if $B$ is division and $\scrO$ is the unique  maximal order,
then $e(\scrO)=-1$.  It is shown in
\cite[Proposition~2.1]{Brzezinski-1983} that $e(\scrO)=1$ if and only if $\scrO$ is a non-maximal Eichler order. Note that $e(\scrO)=1$ can only occur  when $B\simeq \Mat_2(F)$. Moreover, if $e(\scrO)\neq 0$, then $\scrO$ is
automatically Bass by \cite[Corollary~2.4 and
Propoisition~3.1]{Brzezinski-1983}.  The classification of lattices
over Eichler orders is well known (see
\cite[p.~315]{xue-yang-yu:sp_as2} for example), which we recall as follows.

\begin{lem}
  Suppose $B\simeq \Mat_2(F)$ and let $\scrO\simeq \begin{bmatrix} O_F & O_F \\
  \varpi^e O_F & O_F\end{bmatrix}$ be an Eichler order. Let $M$ be an $\scrO$-lattice in a finite left $B$-module $W$. Then
 \begin{equation}
         \label{eq:eichler}
     W\simeq \begin{bmatrix}
            F \\ F
          \end{bmatrix}^{\oplus u}, \quad \text{and} \quad  
          M\simeq \bigoplus_{i=1}^{u}
          \begin{bmatrix}
            O_F \\ \varpi^{e_i} O_F
          \end{bmatrix},
        \end{equation}
where the $e_i$'s are integers such that  $0 \le e_i \le e$ and $e_{i}\le e_{i+1}$ for all $i$. Moreover, the isomorphism class of $M$ is uniquely determined by these $e_i$'s.  
\end{lem}

Henceforth we assume
that $e(\scrO)\in \{0, -1\}$. 
Let $n(\scrO)$ be the unique non-negative integer such that
$\grd(\scrO)=(\varpi^{n(\scrO)})$. Suppose that 
 $\scrO$ is Bass but non-hereditary. From \cite[Proposition~1.12]{Brzezinski-1983}, $\scrO$
has a unique minimal overorder $\calM(\scrO)$, which is also Bass by
definition. According to  \cite[Propositions~3.1 and
4.1]{Brzezinski-1983},
\begin{equation}
  \label{eq:40}
  n(\calM(\scrO))=
  \begin{cases}
    n(\scrO)-2 &\text{if } e(\scrO)=-1,\\
    n(\scrO)-1 &\text{if } e(\scrO)=0,
  \end{cases}
\end{equation}
and $e(\calM(\scrO))=e(\scrO)$ if 
 $\calM(\scrO)$ is also non-hereditary.  Thus starting from
$\calM^0(\scrO):=\scrO$, we  define $\calM^i(\scrO):=\calM(\calM^{i-1}(\scrO))$ recursively
to obtain a unique chain of Bass orders  terminating at a hereditary order
$\calM^m(\scrO)$:
\begin{equation}\label{eq:19}
 \scrO=\calM^0(\scrO)\subset \calM^1(\scrO)\subset
  \calM^2(\scrO)\subset \cdots \subset \calM^{m-1}(\scrO)\subset
  \calM^m(\scrO),   
\end{equation}
where  $m$ is given as follows
\begin{itemize}
\item $m=n(\scrO)-1$ if $e(\scrO)=0$; and 
\item $m=\fl{n(\scrO)/2}$ if $e(\scrO)=-1$, where $x\mapsto \fl{x}$ is the
  floor function on $\R$.
\end{itemize}
The order $\calM^m(\scrO)$ is called the \emph{hereditary closure} of
$\scrO$ and will henceforth be denoted by $\calH(\scrO)$. If $e(\scrO)=-1$, then
$\calH(\scrO)$ is always a maximal order by
\cite[Proposition~3.1]{Brzezinski-1983}. 
Thus when $e(\scrO)=-1$,
$n(\scrO)$ is even if $B\simeq \Mat_2(F)$, and $n(\scrO)$ is odd if
$B$  is division. If $e(\scrO)=0$, then
\begin{itemize}
\item $\calH(\scrO)\simeq \left[\begin{smallmatrix}
  O_F & O_F\\ \varpi O_F & O_F
\end{smallmatrix}\right]$  if $B\simeq \Mat_2(F)$, and 
\item $\calH(\scrO)$ is the unique maximal order if $B$ is division. 
\end{itemize}
Note that $\scrO$ is hereditary (i.e.~$m=0$) if and only if $e(\scrO)=-1$ and
$B$ is division, so $m$ is strictly positive in the remaining cases. 

From \cite[Proposition~1.12]{Brzezinski-crelle-1990}, there
exists a quadratic field extension
$L/F$  such that $O_L$ embeds\footnote{From the proof of \cite[Theorem~3.3 and
  3.10]{Brzezinski-crelle-1990}, any two embeddings of 
  $O_L$ into $\scrO$ are conjugate by an element of the
  normalizer of $\scrO$, thus expression (\ref{eq:15}) does not
  depend on the choice of the embedding $O_L\hookrightarrow \scrO$.} into $\scrO$, and 
\begin{align}
  \scrO&=O_L+\grJ(\calH(\scrO))^c,  \qquad \text{where} \label{eq:15}\\
  \label{eq:14}
  c&=
  \begin{cases}
    n(\scrO)/2 &\text{if } e(\scrO)=-1 \text{ and } B\simeq \Mat_2(F),\\
    n(\scrO)-1&\text{otherwise.}
  \end{cases}
\end{align}
In fact,  $L/F$ is the unique unramified quadratic field extension if
$e(\scrO)=-1$, and it is a ramified quadratic field extension if
$e(\scrO)=0$. In the latter case, the ramified quadratic extension $L/F$ can be arbitrary
if $n(\scrO)=2$
according to \cite[(3.14)]{Brzezinski-crelle-1990}; and it is uniquely
determined by $\scrO$ if $n(\scrO)\geq 3$ and $F$ is nondyadic
according to \cite[Lemma~3.5]{peng-xue:select}. 

\begin{lem}\label{lem:indecomposable}
Suppose that $\scrO\subset B$ is a Bass order with $e(\scrO)\in \{0,
-1\}$. 
Let $N$ be an indecomposable left $\scrO$-lattice. 
\begin{enumerate}
\item If $B$ is division, then
  \begin{equation}
    \label{eq:16}
    N\simeq \calM^i(\scrO) \qquad \text{for some } 0\leq i\leq m. 
  \end{equation}
\item Suppose that $B$  is split, i.e.~$B\simeq \Mat_2(F)$. Fix an identification of 
   $\calH(\scrO)$ with $\left[\begin{smallmatrix}
  O_F & O_F\\ \varpi O_F & O_F
\end{smallmatrix}\right]$ (resp.~$\Mat_2(O_F)$) if $e(\scrO)=0$ (resp.~$-1$).  

  \begin{enumerate}
  \item   If $e(\scrO)=0$, then $N$ is isomorphic to one of the
    following $\scrO$-lattices: 
    \begin{equation}
\label{eq:22}
\begin{bmatrix}
        O_F \\ \varpi O_F
      \end{bmatrix}, \quad        \begin{bmatrix}
        O_F \\ O_F 
      \end{bmatrix}, \quad  \text{or}\quad \calM^i(\scrO)\quad
    \text{with}\quad 0\leq i
    \leq m-1.
  \end{equation}
    \item If $e(\scrO)=-1$, then $N$ is isomorphic to one of the
    following $\scrO$-lattices: 
    \begin{equation}
      \label{eq:17}
      \begin{bmatrix}
        O_F \\ O_F 
      \end{bmatrix} \quad \text{or}\quad  \calM^i(\scrO)\quad
    \text{with}\quad 0\leq i
    \leq m-1.
    \end{equation}
  \end{enumerate}
\end{enumerate}
\end{lem}
\begin{proof}
According to  the Drozd-Krichenko-Roiter Theorem
  \cite[Theorem~37.16]{curtis-reiner:1},
  \begin{equation}
    \label{eq:18}
    N\otimes_{O_F}F\simeq
    \begin{cases}
      B &\text{if $B$ is division},\\
      F^2 \text{ or } \Mat_2(F) &\text{if } B\simeq \Mat_2(F).
    \end{cases}
  \end{equation}

  First, suppose that $B\simeq \Mat_2(F)$ and $N\otimes_{O_F}F\simeq
F^2$. Then $N\simeq O_F^2$ as an $O_F$-module, and $\End_{O_F}(N)$ is
a maximal order in $B$ containing $\scrO$. It follows from
(\ref{eq:19}) that $\End_{O_F}(N)$ contains $\calH(\scrO)$, which 
equips  $N$ with a canonical  $\calH(\scrO)$-module
structure. Therefore, if $e(\scrO)=-1$, then 
$\calH(\scrO)=\Mat_2(O_F)$, and hence $N$ is homothetic to $\left[
  \begin{smallmatrix}
    O_F \\  O_F
  \end{smallmatrix}\right]$. Similarly, if
$e(\scrO)=0$, then $\calH(\scrO)=\left[\begin{smallmatrix}
  O_F & O_F\\ \varpi O_F & O_F
\end{smallmatrix}\right]$, and hence $N$ is homothetic 
to $\left[
  \begin{smallmatrix}
    O_F \\  O_F
  \end{smallmatrix}\right]$ or $\left[
  \begin{smallmatrix}
    O_F \\ \varpi O_F
  \end{smallmatrix}\right]$. 

Next, suppose that $N\otimes_{O_F}F\simeq B$.  Then we regard $N$ as a
fractional left ideal of $\scrO$. Let
$O_l(N)=\{x\in B\mid xN\subseteq N\}$ be the associated left order of
$N$. Clearly, $O_l(N)$ contains $\scrO$, so $O_l(N)=\calM^i(\scrO)$ for some $0\leq i\leq m$. In particular, $O_l(N)$ is
Gorenstein. It follows from Lemma~\ref{lem:Gorenstein} that $N\simeq
O_l(N)$ as $\scrO$-lattices. 

Clearly, if $B$ is division, then $\calM^i(\scrO)$ is
indecomposable for every $0\leq i \leq m$. On the other hand, 
if $B\simeq \Mat_2(F)$, then the hereditary closure
$\calH(\scrO)=\calM^m(\scrO)$ is
\emph{decomposable} as an $\scrO$-lattice. Thus $N\not\simeq
\calM^m(\scrO)$ in this case. It remains to show that
$\calM^i(\scrO)$ is indecomposable for the remaining  $i$'s. Suppose
otherwise so that $\calM^i(\scrO)=N_1\oplus N_2$, where each $N_j$ is an
$\scrO$-lattice in $N_j\otimes_{O_F}F\simeq F^2$.  Then
\[\calM^i(\scrO)=O_l(\calM^i(\scrO))=O_l(N_1\oplus N_2)=O_l(N_1)\cap O_l(N_2). \]
Since $O_l(N_i)$ is a maximal order in $\Mat_2(F)$ for each $i$, this would
imply that $\calM^i(\scrO))$ is an Eichler order
(i.e.~$e(\calM^i(\scrO))\in \{1, 2\}$), 
contradicting to the fact
that $e(\calM^i(\scrO))=e(\scrO)\in \{0, -1\}$ for $0\leq i\leq m-1$. 
  \end{proof}

  Applying the Krull-Schmidt-Azumaya
  Theorem~\cite[Theorem~6.12]{curtis-reiner:1}, we immediately obtain the
  following proposition. 
  \begin{prop}\label{cor:lattices-over-Bass}
    Suppose that $\scrO\subset B$ is a Bass order with $e(\scrO)\in \{0,-1\}$. 
    Let $M$ be an $\scrO$-lattice in a finite left
     $B$-module $W$. 


    \begin{enumerate}
    \item If $B$ is division, then
      $M\simeq \bigoplus_{i=0}^m \calM^i(\scrO)^{\oplus t_i}$ with
      $(t_0, \cdots, t_m)\in \Z_{\geq 0}^{m+1}$ and  $\sum_{i=0}^m t_i=\dim_BW$.
    \item If $B\simeq \Mat_2(F)$, then $W\simeq \Mat_{2, u}(F)$ for
      some  $u\geq 0$. There are two cases to consider: 
      \begin{enumerate}
      \item[(2a)] if $e(\scrO)=0$, then
        \begin{equation}
         \label{eq:23}
          M\simeq
          \begin{bmatrix}
            O_F \\ \varpi O_F
          \end{bmatrix}^{\oplus r}\bigoplus           \begin{bmatrix}
            O_F \\ O_F
          \end{bmatrix}^{\oplus s}\bigoplus
          \bigoplus_{i=0}^{m-1}\calM^i(\scrO)^{\oplus t_i} 
        \end{equation}
        with $(r, s, t_0, \cdots, t_{m-1})\in \Z_{\geq 0}^{m+2}$ and
        $r+s+2\sum_{i=0}^{m-1}t_i=u$; 
        \item[(2b)] if $e(\scrO)=-1$, then
        \begin{equation}
          \label{eq:20}
          M\simeq
                    \begin{bmatrix}
            O_F \\ O_F
          \end{bmatrix}^{\oplus s}\bigoplus
          \bigoplus_{i=0}^{m-1}\calM^i(\scrO)^{\oplus t_i}, 
        \end{equation}
        with $(s, t_0, \cdots, t_{m-1})\in \Z_{\geq 0}^{m+1}$ and
        $s+2\sum_{i=0}^{m-1}t_i=u$. 
      \end{enumerate}
    \end{enumerate}
In all cases, the isomorphism class of $M$ is uniquely determined by
the numerical invariants $r, s$ (if applicable) and the $t_i$'s. 
  \end{prop}

\subsection{Explicit computations}
Recall that our goal is to compute the value of $o(n)$ for $n \in \{3,4,5,8,12\}$
and $p | n$.  As explained in Section~\ref{sec:gen-strategy}, $o(n)$
coincides with the number of isomorphism classes of $\scrA_n$-lattices in
the left $\scrK_n$-module $V=D^2$ (See \eqref{eq:5} and \eqref{eq:7} for the definition of $\scrA_n$ and $\scrK_n$). 
From
\cite[Theorem~11.1]{Washington-cyclotomic}, the $n$-th cyclotomic
field $K_n$ has class number $1$ for each $n\in \{3,4,5,8,12\}$.

Since $K_n$ is totally imaginary and $p$ does not split
completely in $K_n$,  we have  $\scrK_n=K_n\otimes_\Q
D=\Mat_2(K_n)$.  Thus as a left $\Mat_2(K_n)$-module, 
\begin{equation}\label{eq:8}
  V\simeq 
  \begin{cases}
    \Mat_2(K_n) &\text{if } n\in \{3, 4\},\\
    K_n^2&\text{if } n\in \{5, 8, 12\}. 
  \end{cases}
\end{equation}
From this, we can easily write down its endomorphism algebra
\begin{equation}
  \scrE_n=\End_{\scrK_n}(V)=  \begin{cases}
    \Mat_2(K_n) &\text{if } n\in \{3, 4\},\\
    K_n&\text{if } n\in \{5, 8, 12\}.
  \end{cases}
\end{equation}
In particular, we see that every arithmetic subgroup of $\scrE_n^\times$ is infinite, and hence the abelian surfaces in these isogeny classes have infinite automorphism groups.   

Given  an $\scrA_n$-lattice  $\Lambda\subset V$, its endomorphism ring
$O_\Lambda=\End_{\scrA_n}(\Lambda)$ is an  $A_n$-order in
$\scrE_n$. Therefore, if $n\in \{5, 8, 12\}$, then
$O_\Lambda=A_n$ and $h(O_\Lambda)=1$. Now suppose that $n\in \{3,4\}$. We are going to show in \eqref{eq:25} that $\det(O_{\Lambda,
  p}^\times)=A_{n, p}^\times$.  On the other hand, at each prime
$\ell\neq p$, we have $O_{\Lambda, \ell}\simeq
\Mat_2(A_{n, \ell})$ since $O_{\Lambda, \ell}$ is maximal. Thus if we
write $\wh O_\Lambda$ (resp.~$\wh A_n$) for the profinite completion of
$O_\Lambda$ (resp.~$A_n$),
then $\det(\wh O_\Lambda^\times)=\wh{A}_n^\times$. The same proof of
\cite[Corollaire~III.5.7(1)]{vigneras} shows that
$h(O_\Lambda)=h(A_n)=1$. In conclusion, for every pair $(n, p)$ with
$n\in \{3, 4, 5, 8, 12\}$ and $p|n$, we have 
\begin{equation}
  \label{eq:13}
  o(n)=\abs{\scrL_p(n)},  
\end{equation}
which is consistent with \cite[(4.3)]{xue-yang-yu:sp_as2}.

\begin{lem}\label{lem:simple-module-ram}
  Suppose that $n\in \{3, 4, 5, 8\}$ and $p|n$. Let $U_p:=K_{n,p}^2$ be the
  unique simple $\scrK_{n, p}$-module. Then up to isomorphism, 
 there is a unique $\scrA_{n, p}$-lattice  in $U_p$. 
\end{lem}
\begin{proof}
For each pair $(n, p)$ under consideration, $p$ is totally ramified in
$K_n$. It follows from \cite[\S2.4]{li-xue-yu:unit-gp} that the
Eichler invariants 
\begin{equation}\label{eq:48}
 e(\scrA_{n,p})=e(\calO_p)=-1. 
\end{equation}
  From
\cite[Proposition~3.1]{Brzezinski-1983}, $\scrA_{n,p}$ is a Bass
order.  Thus the lemma is a direct application of 
Corollary~\ref{cor:lattices-over-Bass}. 
\end{proof}

\begin{lem}
Suppose that $n\in \{3, 4, 5, 8\}$ and $p|n$. Then
  \begin{equation}
    \label{eq:21}
    o(n)=
    \begin{cases}
      2 &\text{if } n\in \{3, 4\},\\
      1&\text{if } n\in \{5, 8\}.\\
    \end{cases}
  \end{equation}
\end{lem}
\begin{proof}

First, suppose that $n\in \{5, 8\}$. Then $V_p\simeq K_{n,p}^2$ is a
simple $\Mat_2(K_{n, p})$-module.  From
Lemma~\ref{lem:simple-module-ram}, $\abs{\scrL_p(n)}=1$, and hence
$o(n)=1$ by  (\ref{eq:13}).

%

Next, suppose that $n\in \{3, 4\}$. We have already seen in the proof
of Lemma~\ref{lem:simple-module-ram} that $\scrA_{n, p}$ is a Bass
order with Eichler invariant $-1$.   Let $\varpi_n=1-\zeta_n$ be the
uniformizer of the local field $K_{n, p}$. The reduced discriminant
$\grd(\scrA_{n, p})$ is given by 
\begin{equation}
  \label{eq:28}
 \grd(\scrA_{n, p})=\grd(A_{n, p}\otimes_{\Z_p}\calO_p)=\grd(\calO_p)A_{n,p}=pA_{n,p}=\varpi_n^2A_{n,p}.   
\end{equation}
Thus the chain of Bass orders in (\ref{eq:19}) reduces to $\scrA_{n,
  p}\subset \calM(\scrA_{n,
  p})$, where $\calM(\scrA_{n,
  p})=\Mat_2(A_{n,p})$ under a suitable identification $\scrK_{n,
  p}=\Mat_2(K_{n, p})$. In this case, $V_p$ is a free
$\Mat_2(K_{n,p})$-module of rank $1$. From (\ref{eq:20}), every
$\scrA_{n, p}$-lattice $\Lambda_p$ in $V_p$ is isomorphic to
either $\scrA_{n, p}$ or $\Mat_2(A_{n,p})$.  Correspondingly, the
endomorphism ring $O_{\Lambda, p}$ is given by 
\begin{equation}
  \label{eq:24}
  O_{\Lambda,p}=\End_{\scrA_{n, p}}(\Lambda_p)\simeq
  \begin{cases}
    \scrA_{n, p} &\text{if } \Lambda_p\simeq \scrA_{n, p},\\
    \Mat_2(A_{n, p}) &\text{if } \Lambda_p\simeq \Mat_2(A_{n, p}).\\
  \end{cases}
\end{equation}
In both cases, we have
\begin{equation}
  \label{eq:25}
\det(O_{\Lambda,
  p}^\times)=A_{n, p}^\times.   
\end{equation}
Indeed, this is clear if $O_{\Lambda,p}\simeq \Mat_2(A_{n, p})$. In
the case $O_{\Lambda,p}\simeq \scrA_{n, p}$, let $L$ be the unique
unramified  quadratic field extension of $K_{n, p}$. From (\ref{eq:15}),
$O_{\Lambda,p}$ contains a copy of $O_L$, which implies that $\det(O_{\Lambda,
  p}^\times)\supseteq \Nm_{L/K_{n,p}}(O_L^\times)=A_{n,
  p}^\times$. On the other hand, $\det(O_{\Lambda,
  p}^\times)$ is obviously contained in $A_{n, p}^\times$, so  equality (\ref{eq:25}) holds in this case as
well. We conclude that $o(n)=\abs{\scrL_p(n)}=2$ if $n\in \{3, 4\}$
and $p|n$. 
\end{proof}

\begin{lem}
  $o(12)=3$ if $p\in \{2,3\}$. 
\end{lem}
\begin{proof}
Since $2$ and $3$ are ramified in $K_n$, $2$ is inert in $\Q(\zeta_3)$ and $3$ is inert in $\Q(\zeta_4)$, the $p$-adic completion $K_{n, p}$ is a field extension of degree $4$ over
  $\Q_p$ with residue degree $2$, so
  $e(\scrA_{n, p})=e(A_{n,p}\otimes_{\Z_p}\calO_p)=1$ by
  \cite[\S2.4]{li-xue-yu:unit-gp}. A similar calculation as 
  (\ref{eq:28}) shows that $\grd(\scrA_{n,p})=\varpi_p^2A_{n,p}$, where 
$\varpi_p$ denotes a  uniformizer  of $K_{n, p}$.  From
\cite[Proposition~2.1]{Brzezinski-1983}, we may identify  $\scrA_{n,
  p}$ with the
Eichler order $\begin{bmatrix}
     A_{n, p} & A_{n, p}\\
 \varpi_p^2A_{n, p}& A_{n, p}
   \end{bmatrix}$.
Since $V_p\simeq K_{n,p}^2$ is a
simple $\Mat_2(K_{n, p})$-module, every $\scrA_{n,
  p}$-lattice in $V_p$ is isomorphic to one of the following lattices: 
\begin{equation}
  \label{eq:26}
  \begin{bmatrix}
    A_{n, p}\\ \varpi_p^2A_{n, p}
  \end{bmatrix},\qquad   \begin{bmatrix}
    A_{n, p}\\\varpi_p A_{n, p}
  \end{bmatrix},\qquad   \begin{bmatrix}
    A_{n, p}\\ A_{n, p}
  \end{bmatrix}.
\end{equation}
Therefore, $o(12)=\abs{\scrL_p(12)}=3$ by (\ref{eq:13}).  
\end{proof}

\begin{rem}
We have proved (cf.~\eqref{eq:13}) that in the isotypic case, every genus in the set $\scrL (n)$ of lattice classes has class number one. This holds also in the case $p\nmid n$ according to  \cite[(4.3)]{xue-yang-yu:sp_as2}.
\end{rem}
  

\section{The non-isotypic case}
\label{sec:non-isotypic-case}
In this section, we compute the values of $o(\uln)$ for  $\uln=(n_1, n_2)\in
\brN^2$ and $p$ satisfying condition (C1) or (C2) (or both). More explicitly, 
the pairs
$(\uln, p)$ are listed in the following
table:
\begin{center}
\begin{tabular}{*{7}{|>{$}c<{$}}|}
\hline
  \uln=(n_1, n_2) & (1, 2) & (2, 3) & (2, 4) & (2, 6) & (3,
                                                                  4) &
                                                                       (3,
                                                                       6)
  \\
\hline
p &  2 & 3 &2 & 3 & 2, 3& 2 , 3\\
  \hline
  i(\uln)& 2 & 1 & 2 & 3 & 1 & 4\\
  \hline
\end{tabular}
\end{center}
Here we have also included the index $i(\uln)=[O_{K_\uln}:A_{\uln}]$,
where $O_{K_\uln}=A_{n_1}\times A_{n_2}$ is the unique maximal order
of $K_\uln$.

Since $K_{\uln}=K_{n_1}\times K_{n_2}$, we have
$\scrK_\uln=\scrK_{n_1}\times \scrK_{n_2}$. Consequently, the left
$\scrK_\uln$-module $V=D^2$ decomposes into a product $V_{n_1}\times V_{n_2}$, where each $V_{n_i}$ is a simple left
$\scrK_{n_i}$-module  with $\dim_{D}V_{n_i}=1$.  In turn, 
$\scrE_\uln=\scrE_{n_1}\times \scrE_{n_2}$ with
$\scrE_{n_i}:=\End_{\scrK_{n_i}}(V_{n_i})$. 
If $n_i\in \{1, 2\}$,
then $K_{n_i}=\Q$, so we have 
\begin{equation}
  \label{eq:30}
  \scrK_{n_i}=D, \quad V_{n_i}\simeq D, \quad \scrE_{n_i}\simeq D
  \quad\text{if} \quad n_i\in \{1, 2\}. 
\end{equation}
If $n_i\in
\{3, 4, 6\}$, then $K_{n_i}$ is an imaginary quadratic extension of
$\Q$, and $p$ does not split completely in $K_{n_i}$. Thus
\begin{equation}
  \label{eq:29}
 \scrK_{n_i}\simeq \Mat_2(K_{n_i}),\quad  V_{n_i}\simeq K_{n_i}^2,
 \quad \scrE_{n_i}=K_{n_i} \quad\text{if} \quad n_i\in
\{3, 4, 6\}.
\end{equation}
To avoid conflict of notations between $V_{n_i}$ and $V_\ell$,  we will always write the full expression
$V\otimes \Q_\ell$ instead of $V_\ell$ for
the $\ell$-adic completion of $V$. On the other hand, the subscript 
$_\uln$ will never be expanded out explicitly as $_{(n_1, n_2)}$ nor
$_{n_1, n_2}$, so there
should be no ambiguity about $\scrA_{n_i,
  \ell}:=\scrA_{n_i}\otimes_\Z\Z_\ell $. 

If $\ell\in \bbN$ is a prime with $\ell\nmid i(\uln)$ and $\ell\neq
p$, then $A_{\uln, \ell}=O_{K_\uln, \ell}$, and $\calO_\ell\simeq
\Mat_2(\Z_\ell)$, which implies that 
\[\scrA_{\uln, \ell}=A_{\uln, \ell}\otimes_{\Z_\ell}\calO_\ell\simeq
\Mat_2(O_{K_\uln, \ell}).\]
 Thus $\scrA_{\uln, \ell}$ is
maximal in $\scrK_{\uln, \ell}$ for such an  $\ell$. From this, we can 
easily write the set  $S(\uln, p)$ of primes at which $\scrA_\uln$
is non-maximal: 
\begin{equation}
  \label{eq:31}
  S(\uln, p)=
  \begin{cases}
    \{2, 3\} &\text{if } \uln=(3, 6)\text{ and } p=3;\\
    \{p\}&\text{otherwise}. 
  \end{cases}
\end{equation}
Recall that the class number $h(\calO)$ is given by the following
formula \cite[Proposition~V.3.2]{vigneras}
\begin{equation}
  \label{eq:32}
  h(\calO)=\frac{p-1}{12}+\frac{1}{3}\left(1-\Lsymb{-3}{p}\right)+\frac{1}{4}\left(1-\Lsymb{-4}{p}\right),  
\end{equation}
where $\Lsymb{\cdot}{p}$ denotes the Legendre symbol. In particular,
$h(\calO)=1$ if $p=2,3$. We also note that $h(A_{n_i})=1$ for every
$n_i\in \{3, 4, 6\}$. Given $d\in \bbN$,  we write $\Q_{\ell^d}$ for the
unique unramified extension of degree $d$ over $\Q_\ell$, and
$\Z_{\ell^d}$ for its ring of integers.


\begin{lem}
$o(2,3)=1$ if $p=3$, and $o(3, 4)=2$ if $p\in \{2, 3\}$. 
\end{lem}
\begin{proof}
For $\uln=(2, 3)$ or $(3, 4)$, we have
  $A_\uln=A_{n_1}\times A_{n_2}$, and hence
  $\scrA_\uln=\scrA_{n_1}\times \scrA_{n_2}$. Any
  $\scrA_{\uln}$-lattice $\Lambda\subset V$ decomposes into a
  product $\Lambda_{n_1}\times \Lambda_{n_2}$, where each
  $\Lambda_{n_i}$ is an $\scrA_{n_i}$-lattice in $V_{n_i}$.  Thus the
  techniques developed in Section~\ref{sec:isotypic-case} applies
  here. 

  
  First suppose that $\uln=(2, 3)$ and $p=3$.  In this case,
  $\scrA_{n_1}=\Z\otimes_\Z \calO=\calO$, and $V_{n_1}\simeq D$ by
  (\ref{eq:30}). Since $h(\calO)=1$, there is a unique isomorphism
  class of $\calO$-lattices in $V_{n_1}$.  As $S(\uln, p)=\{3\}$ in
  this case, we consider the $\scrA_{n_2,
    3}$-lattices in $V_{n_2}\otimes
  \Q_3\simeq K_{n_2, 3}^2$.  From
  Lemma~\ref{lem:simple-module-ram},  there is a unique
  $\scrA_{n_2, 3}$-lattice up to isomorphism in $V_{n_2}\otimes
  \Q_3$.  
 This implies that   there  is only a
single genus of $\scrA_{n_2}$-lattices in $V_{n_2}$.  Note that
  $\End_{\scrA_{n_2}}(\Lambda_{n_2})=A_{n_2}\simeq\Z[\zeta_3]$, which
  has class number $1$. Thus there is a unique isomorphism class of
  $\scrA_{n_2}$-lattices in $V_{n_2}$. As a result,
  $o(2,3)=1\cdot 1=1$ if $p=3$.

  Next, suppose that $\uln=(3, 4)$ and $p=2$. Since $2$ is ramified in
  $K_{n_2}=\Q(\zeta_4)$, the same proof as
  above shows that there is a unique isomorphism class of
  $\scrA_{n_2}$-lattices in $V_{n_2}$ in this case. On the other hand,
  $2$ is inert in $K_{n_1}=\Q(\zeta_3)$, so 
  $A_{n_1, 2}=\Z_4$.  It follows from
  \cite[Lemma~2.10]{li-xue-yu:unit-gp} that
  \begin{equation}
    \label{eq:33}
\scrA_{n_1,2}=A_{n_1, 2}\otimes_{\Z_2}\calO_2=\Z_4\otimes_{\Z_2}\calO_2\simeq 
    \begin{bmatrix}
      \Z_4& \Z_4\\ 2\Z_4 & \Z_4
    \end{bmatrix}.
  \end{equation}
Hence every $\scrA_{n_1, 2}$-lattices in
$V_{n_1}\otimes_\Q \Q_2$ is isomorphic to either $\left[\begin{smallmatrix}
  \Z_4\\ 2\Z_4
\end{smallmatrix}\right]$ or $\left[\begin{smallmatrix}
  \Z_4\\ \Z_4
\end{smallmatrix}\right]$. 
We find that there are two genera of $\scrA_{n_1}$-lattices in
$V_{n_1}$, each consisting of a unique isomorphism class since
$h(A_{n_1})=h(\Z[\zeta_3])=1$. Therefore, $o(3, 4)=(1+1)\cdot 1=2$ if
$p=2$. 

Lastly, the value of $o(3, 4)$ for $p=3$ can be computed in exactly the same
  way as above, since $3$ is ramified in $K_{n_1}$ and inert in
  $K_{n_2}$. 
\end{proof}

\begin{lem}\label{lem:o-1-2}
$o(1,2)=3$ if $p=2$. 
\end{lem}
\begin{proof}
Set $\uln=(1,2)$ throughout this proof. 
  In this case,  $\scrA_{\uln}$ is non-maximal only at $p=2$, and $\calO_2$ is the unique
maximal order in the division quaternion $\Q_2$-algebra
$D_2$. 
  From \cite[(5.4)]{xue-yang-yu:sp_as2}, we have 
\[A_\uln=\{(a,b)\in \Z\times \Z\mid a\equiv b\pmod{2}\}, \]
which implies that 
\begin{equation}
  \label{eq:34}
  \scrA_{\uln, 2}=A_\uln\otimes_\Z\calO_2=\{(x,y)\in \calO_2\times \calO_2\mid x\equiv y\pmod{2\calO_2}\}.
\end{equation}
In particular, $  \scrA_{\uln, 2}$ is a subdirect sum\footnote{Let $\{R_i\mid 1\leq i\leq n\}$ be a finite set of (unital) rings, and
$R:=\prod_{i=1}^n R_i$ be their direct product. 
A
ring $T$ is called a \emph{subdirect sum} of the $R_i$'s if there
exists an embedding $\rho: T\to R$ such that every canonical
projection $\pr_i: R\to R_i$ maps $\rho(T)$ surjectively onto
$R_i$.}   of two copies of $\calO_2$, so it is a Bass order by 
\cite[Proposition~12.3]{Drozd-Kirichenko-Roiter-1967}. 

Let $\grP$ be the unique two-sided prime ideal of $\calO$ above
$p=2$. We put 
\begin{equation}
  \label{eq:35}
  \scrR:=\{(x,y)\in \calO\times \calO\mid x\equiv
  y\pmod{\grP}\}, 
\end{equation}
which has index $4$ in
the maximal order $\bbO:=\calO\times\calO$ in
$\scrK_{\uln}=D\times D$.  Indeed, $\scrR/(\grP\times
\grP)\simeq \calO/\grP\simeq 
\F_4$ while $\bbO/(\grP\times \grP)\simeq \F_4\times \F_4$. 
Clearly, both
$\scrR$ and $\bbO$ are overorders of $\scrA_{\uln}$. We
claim that there are no other overorders except $\scrA_{\uln}$ itself. 
It is enough to prove this locally at $p=2$. From $\scrA_{\uln,2}/(2\calO_2\times
2\calO_2)\simeq \calO_2/2\calO_2$, we find that 
\begin{equation}
  \label{eq:36}
\scrA_{\uln,2}/\grJ(\scrA_{\uln,2})\simeq
\calO_2/\grJ(\calO_2)\simeq \F_4,    
\end{equation}
where $\grJ(\cdot)$ denotes the Jacboson radical. Hence $ \scrA_{\uln, 2}$
is completely primary in the sense of
\cite[p.~262]{Roggenkamp-Latt-II}. Now according to
\cite[Lemma~6.6]{Roggenkamp-Latt-II}, every non-maximal completely
primary Gorenstein order has a unique minimal
overorder. As $ \scrA_{\uln, 2}$ has index
$4$ in $\scrR_2$,  the left $\scrA_{\uln,2}$-module
$\scrR_2/\scrA_{\uln, 2}$ is isomorphic to the unique simple left
$\scrA_{\uln,2}$-module $\F_4$.  Thus $\scrR_2$ coincides with
the unique minimal overorder of $\scrA_{\uln,2}$.   
Similarly, $\scrR_2$ is also completely
primary, and $\bbO_2$ is the unique minimal overorder
of $\scrR_2$ by the same argument.  This verifies the claim about the
overorders of $\scrA_{\uln}$.

In this case, $V\otimes\Q_2$ is a free left module of rank $1$ over
$\scrK_{\uln,2}$.  For any
$\scrA_{\uln,2}$-lattice in $V\otimes \Q_2$, its associated left
order  necessarily coincides with one of the three
overorders of $\scrA_{\uln,2}$. Since $\scrA_{\uln,2}$ is completely primary, it
is indecomposable as a left module over itself. Taking into account
that $\scrA_{\uln, 2}$ is Gorenstein, it follows from
\cite[Proposition~2.3]{Brzezinski-loc-Princ} that every proper
$\scrA_{\uln,2}$-lattice in the $V\otimes \Q_2$ is principal. This holds for
$\scrR_2$ as well by the same token. On the other hand, every proper
$\bbO_2$-lattice in the $V\otimes \Q_2$ is principal  since
$\bbO_2$ is maximal. Therefore,  every
$\scrA_{\uln,2}$-lattice in $V\otimes \Q_2$ is isomorphic to one of the
following
\begin{equation}
  \label{eq:38}
\scrA_{\uln,2}, \qquad \scrR_2, \qquad \bbO_2. 
\end{equation}
We find that there are three genera of $\scrA_\uln$-lattices in
$V$ represented by $\scrA_\uln, \scrR$ and $\bbO$
respectively. Clearly,
\[  \End_{\scrA_\uln}(\scrA_\uln)=\scrA_\uln^{\mathrm{opp}},\quad   \End_{\scrA_\uln}(\scrR)=\scrR^{\mathrm{opp}},\quad \End_{\scrA_\uln}(\bbO)=\bbO^{\mathrm{opp}}.\]
In each case, the opposite ring can be canonically identified with the
original ring itself,
so we drop the superscript $^{\mathrm{opp}}$ henceforth.

It remains to show that
$h(\scrA_\uln)=h(\scrR)=h(\bbO)=1$. This clearly holds
true for the maximal order $\bbO=\calO\times \calO$ since
$h(\bbO)=h(\calO)\cdot h(\calO)=1\cdot1=1$. If we prove
$h(\scrA_\uln)=1$, then $h(\scrR)=1$ since
$h(\scrA_\uln)\geq h(\scrR)$ by
\cite[(6.1)]{xue-yang-yu:sp_as2}.  Let $\wh\bbO$
(resp.~$\wh\scrA_\uln$) be the profinite completion of $\bbO$
(resp.~$\scrA_\uln$). Since $h(\bbO)=1$, it follows from
\cite[(6.3)]{xue-yang-yu:sp_as2} that
\begin{equation}
  \label{eq:41}
  h(\scrA_\uln)=\abs{\bbO^\times\backslash \wh\bbO^\times/(\wh\scrA_\uln)^\times}. 
\end{equation}
Clearly, $\scrA_{\uln, 2}^\times\supseteq 1+2\bbO_2$ and
$\scrA_{\uln, \ell}^\times=\wh\bbO_\ell^\times$ for every prime
$\ell\neq 2$, so there is an $\wh\bbO^\times$-equivariant
projection
\[(\bbO/2\bbO)^\times\simeq \bbO_2^\times/(1+2\bbO_2)
  \twoheadrightarrow \wh\bbO^\times/(\wh \scrA_\uln)^\times. \] Hence
to prove $ h(\scrA_\uln)=1$, it is enough to show that the canonical
projection $\bbO^\times\to (\bbO/2\bbO)^\times$ is surjective. Since
$\bbO=\calO\times \calO$, this amounts to show that
\begin{equation}
  \label{eq:76}
 \varphi: \calO^\times\to (\calO/2\calO)^\times\quad \text{is surjective}.  
\end{equation}
From \cite[Proposition~V.3.1]{vigneras},
$\calO^\times\simeq \SL_2(\F_3)$, which has order $24$. On the other
hand,
\[\abs{(\calO/2\calO)^\times}=\abs{(\calO/\grP^2)^\times}=\abs{\F_4^\times}\cdot
  \abs{\F_4}=3\cdot 4=12. \] Thus to prove the surjectivity of
$\varphi$, it suffices to show that $\ker(\varphi)=\{\pm1\}$. Since
every $\alpha\in \calO^\times$ has finite group order, this follows from a
well-known lemma of Serre \cite[Theorem,
p.~17--19]{Serre-lem-to-Grothen} (See
also  \cite[Lemma, p.~192]{mumford:av},  \cite{Zarhin-Silverberg:Serre-lem} and
\cite[Lemma~7.2]{karemaker2020mass} for some variations and  generalizations).
Therefore, $h(\scrA_\uln)=1$ as claimed.


In conclusion, we have
\[o(1, 2)=h(\scrA_\uln)+h(\scrR)+h(\bbO)=1+1+1=3.\qedhere\]
\end{proof}

For $\uln=(3, 6)$, the  values of $o(\uln)$ for $p\in \{2, 3\}$ will be calculated
in a few steps. 
\begin{lem}\label{lem:3-6-step1}
  $o(3,6)=\prod_{\ell\in S(\uln,  p)}\abs{\scrL_\ell(\uln)}$ for both $p\in \{2, 3\}$, where the set $S(\uln,  p)$ is given in \eqref{eq:31}.
\end{lem}
\begin{proof}
We identify both $A_3$ and 
$A_6$ with $\Z[\zeta_3]$ via the following maps:
\[\Z[T]/(T^2+T+1)\to
  \Z[\zeta_3],\quad T\mapsto \zeta_3, \quad\text{and}\quad \Z[T]/(T^2-T+1)\to \Z[\zeta_3], \quad T\mapsto
  -\zeta_3  .\]
As a result, the maximal order $O_{K_\uln}$ of $K_\uln$ is identified
with  $\Z[\zeta_3]\times \Z[\zeta_3]$. 
From \cite[(5.7)]{xue-yang-yu:sp_as2}, we have
\begin{equation}
  \label{eq:39}
  A_\uln=\{(a, b)\in O_{K_\uln}\mid a\equiv
  b\pmod{2\Z[\zeta_3]}\}. 
\end{equation}
In particular, $A_\uln/(2O_{K_\uln})\simeq
\F_4$. Hence  $[O_{K_\uln}:
A_\uln]=4$, and  the overorders of $A_\uln$ are precisely $O_{K_\uln}$ and  $A_\uln$ itself. From
(\ref{eq:29}),  $\scrE_\uln=K_{\uln}$, so the endomorphism
ring $O_\Lambda$ of any 
$\scrA_\uln$-lattice $\Lambda\subset V$ is an overorder of
$A_\uln$. Thus  $O_\Lambda$ is equal to either $A_\uln$ or
$O_{K_\uln}$. Since $h(A_\uln)=h(O_{K_{\uln}})=1$ by
\cite[(5.8)]{xue-yang-yu:sp_as2}, we conclude that
\begin{equation}
  \label{eq:42}
  o(3,6)=\prod_{\ell\in S(\uln,  p)}\abs{\scrL_\ell(\uln)}.\qedhere
\end{equation}
\end{proof}

\begin{lem}
$ o(3, 6)=2$ if $p=3$. 
\end{lem}
\begin{proof}
 From (\ref{eq:31}), $S(\uln,
  3)=\{2, 3\}$, so we need to classify $\scrA_{\uln, \ell}$-lattices
  in $V\otimes \Q_\ell$ for both $\ell=2,3$.  For $\ell=2$, we have $\calO_2=\Mat_2(\Z_2)$. Hence
 \begin{equation}
   \label{eq:44}
   \scrA_{\uln, 2}=A_{\uln, 2}\otimes_{\Z_2}\calO_2=\Mat_2(A_{\uln, 2}).
 \end{equation}
Recall that $V=K_\uln^2$ by (\ref{eq:29}). Applying Morita's
equivalence, we  reduce the classification of
$\scrA_{\uln, 2}$-lattices in $V\otimes \Q_2$ to that of $A_{\uln, 2}$-lattices
in $K_{\uln, 2}$. Now $A_{\uln, 2}$ is a commutative Bass order over
$\Z_2$, so it follows from Lemma~\ref{lem:Gorenstein} that each $A_{\uln, 2}$-lattice in $K_{\uln, 2}$ is isomorphic to
an overorder of $A_{\uln, 2}$ (namely, either $A_{\uln, 2}$ or $O_{K_\uln,
  2}$).  Therefore, $\abs{\scrL_2(\uln)}=2$.


Next, we compute $\abs{\scrL_3(\uln)}$.   Since $3$ is coprime to
  $[O_{K_\uln}: A_\uln]=4$,  we have
  \begin{align}
    \label{eq:43}
    A_{\uln, 3}&=O_{K_\uln}\otimes_\Z \Z_3=A_{n_1,
    3}\times A_{n_2, 3},\quad \text{and hence}\\
    \scrA_{\uln, 3}&=A_{\uln, 3}\otimes_{\Z_3}\calO_3=\scrA_{n_1,3}\times
    \scrA_{n_2, 3}.  
  \end{align}
Here $\scrA_{n_i, 3}= \Z_3[\zeta_3]\otimes_{\Z_3}\calO_3$ for each
  $i=1,2$.  On
  the other hand,  $V_{n_i}\otimes \Q_3\simeq K_{n_i, 3}^2$  by (\ref{eq:29}).   It
  follows from Lemma~\ref{lem:simple-module-ram} that
  $\abs{\scrL_3(\uln)}=1$.

  We conclude from Lemma~\ref{lem:3-6-step1} that $o(3, 6)=2\cdot 1=2$
  when $p=3$. 
\end{proof}

\begin{prop}\label{lem:o-3-6-p-2}
  $o(3,6)=8$ if $p=2$. 
\end{prop}
\begin{proof}
  In this case, $S(\uln, 2)=\{2\}$ by (\ref{eq:31}).  Let
  $\Q_4$ be the unique unramified quadratic extension of $\Q_2$,
  and $\Z_4$ be its ring of integers. Then $A_{n_i, 2}=\Z_2[\zeta_3]=\Z_4$ for
  both $i=1,2$.  The same calculation as in (\ref{eq:33}) shows that
  \begin{equation}
    \label{eq:45}
    O_{K_\uln, 2}\otimes_{\Z_2}\calO_2= (\Z_4\times
    \Z_4)\otimes_{\Z_2}\calO_2\simeq 
    \begin{bmatrix}
      \Z_4& \Z_4\\ 2\Z_4 & \Z_4
    \end{bmatrix}\times \begin{bmatrix}
      \Z_4& \Z_4\\ 2\Z_4 & \Z_4
    \end{bmatrix}.
  \end{equation}
  For simplicity, put $\bbE:=\left[
    \begin{smallmatrix}
      \Z_4 & \Z_4\\ 2\Z_4 & \Z_4 
    \end{smallmatrix}\right]$ and identify $(O_{K_\uln,
    2}\otimes_{\Z_2}\calO_2)$ with $\scrB:=\bbE\times \bbE$. Then it
  follows from (\ref{eq:39}) that 
  \begin{equation}
    \label{eq:46}
    \scrA_{\uln, 2}=\{ (x, y)\in \scrB\mid a\equiv
    b\pmod{2\bbE}\}. 
  \end{equation}
In particular, $2\scrB$ is a two-sided ideal of
$\scrA_{\uln, 2}$ contained in $\grJ(\scrA_{\uln, 2})$.  We put
\begin{equation}\label{eq:52}
  \bar\scrA_{\uln, 2}:=\scrA_{\uln,
  2}/(2\scrB)=\bbE/2\bbE,\qquad
\bar\scrB:=\scrB/2\scrB=(\bbE/2\bbE)\times (\bbE/2\bbE),  
\end{equation}
where $\bar\scrA_{\uln, 2}$  embeds into
 $\bar\scrB$ diagonally.

From  (\ref{eq:29}), $V_{n_i}\otimes\Q_2=K_{n_i,
  2}^2=
\left[\begin{smallmatrix}
  \Q_4\\ \Q_4
\end{smallmatrix}\right]
$ for each $i=1, 2$.  For simplicity, we identify $V\otimes \Q_2$ with
$\Mat_2(\Q_4)$ and regard the $i$-th column as a $\scrK_{n_i,
  2}$-module for each $i$. 
Any $\scrB$-lattice in
$V\otimes \Q_2$ is isomorphic to one of the following
\begin{equation}
  \label{eq:49}
\begin{bmatrix}
    \Z_4& \Z_4 \\ 2\Z_4 & 2\Z_4
  \end{bmatrix},\qquad  \begin{bmatrix}
    \Z_4& \Z_4\\
    \Z_4&  \Z_4
  \end{bmatrix},\qquad \begin{bmatrix}
    \Z_4 & \Z_4 \\ 2\Z_4 & \Z_4
  \end{bmatrix},\qquad \begin{bmatrix}
    \Z_4&  \Z_4 \\     \Z_4& 2\Z_4
  \end{bmatrix}.
\end{equation}
From Lemma~\ref{lem:3-6-step1},  we are only concerned with $\scrA_{\uln, 2}$-lattices in
$V\otimes \Q_2$, so for ease of notation we  drop the
subscript $_2$ from $\Lambda_2$ and write $\Lambda$ for  an $\scrA_{\uln, 2}$-lattice
in  $V\otimes\Q_2$.
Replacing $\Lambda$ by $g\Lambda$ for a
suitable $g\in \scrE_{\uln, 2}^\times$ if necessary, we  assume that
$\scrB\Lambda$ is equal to one of the $\scrB$-lattices $\Delta$ in
(\ref{eq:49}).   Clearly    $2\Delta\subseteq \Lambda\subseteq \Delta$ since 
$2\scrB\subseteq \scrA_{\uln, 2}$. Thus $\bar\Lambda:=\Lambda/(2\Delta)$ is an
$\bar\scrA_{\uln, 2}$-submodule of $\bar{\Delta}:=\Delta/2\Delta$. Moreover, 
as a $\bar\scrB$-module, $\bar{\Delta}$ is spanned
by $\bar\Lambda$.  Fix one $\Delta$ in (\ref{eq:49}). 
Let $\grS(\Delta)$ be the set of $\scrA_{\uln, 2}$-sublattices $\Lambda$ of $\Delta$
satisfying $\scrB \Lambda=\Delta$, and $\grS(\bar \Delta)$  
be the set of $\bar\scrA_{\uln,
  2}$-submodules $\bar M$ of $\bar \Delta$ satisfying $\bar\scrB\bar M=\bar
\Delta$.  For any $\bar M\in \grS(\bar \Delta)$, there is a unique $\scrA_{\uln,
  2}$-sublattice  
$\Lambda$ of $\Delta$ satisfying $\Lambda\supseteq 2\Delta$ and 
$\bar\Lambda=\bar M$. Moreover, $\scrB\Lambda=\Delta$ by  Nakayama's lemma.  Hence the association $\Lambda\mapsto \bar
\Lambda$ induces a bijective map: 
\begin{equation}
  \label{eq:50}
  \grS(\Delta)\xrightarrow{\sim} \grS(\bar \Delta), \qquad \Lambda\mapsto \bar \Lambda.
\end{equation}
The group $\End_{\scrB}(\Delta)^\times$ acts on both $\grS(\Delta)$
and $\grS(\bar \Delta)$, and the map in (\ref{eq:50}) 
is $\End_{\scrB}(\Delta)^\times$-equivariant as well.  
Two members $\Lambda, \Lambda'$ of $\grS(\Delta)$ are
$\scrA_{\uln, 2}$-isomorphic if and only if there exists an 
$\alpha\in \End_{\scrB}(\Delta)^\times$ such that $\alpha
\Lambda=\Lambda'$.  Therefore, we have established the following bijection 
\begin{equation}
  \label{eq:51}
  \left\{\parbox{1.6in}{Isomorphism classes of $\scrA_{\uln,
        2}$-lattices $\Lambda$ in $V\otimes \Q_2$ with $\scrB
      \Lambda\simeq \Delta$} \right\}\longleftrightarrow\left\{\parbox{1.8in}{$\End_{\scrB}(\Delta)^\times$-equivalent classes of $\bar\scrA_{\uln, 2}$-submodules $\bar M$ in
$\bar \Delta$ such that $\bar\scrB \bar M=\bar \Delta$}\right\}. 
\end{equation}
Note that $\End_{\scrB}(\Delta)=\Z_4\times
\Z_4$ for every $\Delta$ in (\ref{eq:49}), so the action of
$\End_{\scrB}(\Delta)^\times$ on $\grS(\bar \Delta)$ factors through $\F_4^\times\times
\F_4^\times$.

Recall from (\ref{eq:52}) that $\bar\scrA_{\uln, 2}=\bbE/2\bbE$, and
$\bar\scrB =(\bbE/2\bbE)^2$. 
We describe the $4$-dimensional $\F_4$-algebra $\bbE/2\bbE$ more
concretely. Let $R$ be the commutative $\F_4$-algebra $\F_4\times \F_4$. Regard $Q:=R$ as an $(R, R)$-bimodule such
that for each $(a, d)\in R$ and $(b,c)\in Q$, the multiplications are
given by the following rules:
\begin{equation}
  \label{eq:47}
  (a,d) \cdot (b,c) = (ab,dc), \qquad
  (b,c) \cdot (a,d) = (bd, ca).
\end{equation}
We form the trivial extension \cite[Example~1.14]{Lam-noncom-ring} of
$R$ by $Q$ and denote it as $\bar{\bbE}:=\left<
\begin{smallmatrix}
  \F_4 & \F_4\\ \F_4 & \F_4
\end{smallmatrix}\right>$, where $R$ is identified with the diagonal of $\left<
\begin{smallmatrix}
  \F_4 & \F_4\\ \F_4 & \F_4
\end{smallmatrix}\right>$ and $Q$ is identified with the anti-diagonal.
  More explicitly,  the product in $\bar\bbE$ is defined by the
  following rule: 
  \begin{equation}
    \label{eq:55}
\left<\begin{matrix}
      a & b \\
      c & d
    \end{matrix}\right>
\cdot
    \left<\begin{matrix}
      a' & b' \\
      c' & d'
    \end{matrix}\right>
    := \left<\begin{matrix}
      aa'     & ab'+bd' \\
      ca'+dc' & dd'
    \end{matrix}\right>.
  \end{equation}
For each $x\in \Z_4$, let $\bar{x}$ be its canonical image in
$\F_4=\Z_4/2\Z_4$. 
One easily checks that the following map induces an isomorphism
between $\bbE/2\bbE$ and $\bar\bbE$:
\[\begin{bmatrix}
      \Z_4& \Z_4\\ 2\Z_4 & \Z_4
    \end{bmatrix}\to \left<\begin{matrix}
      \F_4 & \F_4 \\
      \F_4 & \F_4
    \end{matrix}\right>, \qquad \begin{bmatrix}
      x& y\\ 2z & w
    \end{bmatrix}\mapsto \left<\begin{matrix}
      \bar{x} & \bar{y} \\
      \bar{z} & \bar{w}
    \end{matrix}\right>.\]
Henceforth $\bbE/2\bbE$ will be identified with $\bar\bbE$ via this
induced isomorphism.  Each column of $\bar\bbE$  admits a canonical 
$\bar\bbE$-module  structure. These two $\bar\bbE$-modules will be denoted by $\left[\begin{smallmatrix}
  \F_4\\ \F_4
\end{smallmatrix}\right]^\dagger$ and  $\left[\begin{smallmatrix}
  \F_4\\ \F_4
\end{smallmatrix}\right]^\ddagger$ respectively. It is clear from (\ref{eq:55}) that  $\left[\begin{smallmatrix}
 0\\ \F_4
\end{smallmatrix}\right]^\dagger$ (resp.~$\left[\begin{smallmatrix}
   \F_4\\0
\end{smallmatrix}\right]^\ddagger$) is the unique $1$-dimensional (as an
$\F_4$-vector space) $\bar\bbE$-submodule of $\left[\begin{smallmatrix}
  \F_4\\ \F_4
\end{smallmatrix}\right]^\dagger$ (resp.~$\left[\begin{smallmatrix}
  \F_4\\ \F_4
\end{smallmatrix}\right]^\ddagger$). We leave it as a simple exercise
to check 
that $\left[\begin{smallmatrix}
  \F_4\\ \F_4
\end{smallmatrix}\right]^\dagger$ and  $\left[\begin{smallmatrix}
  \F_4\\ \F_4
\end{smallmatrix}\right]^\ddagger$ are non-isomorphic $\bar \bbE$-modules.


Now for each $\Delta$ in (\ref{eq:49}), we classify the $(\F_4^\times\times
\F_4^\times)$-orbits of $\grS(\bar{\Delta})$. For $i=1,2$, let $\Delta_i$ be the
$i$-th column of $\Delta$  so that $\bar{\Delta}=\bar{\Delta}_1\times \bar{\Delta}_2$. Let
$\bar M\subseteq \bar{\Delta}$ be an $\bar\bbE$-submodule, and
$\pr_i: \bar{M}\to \bar{\Delta}_i$ be the projection map to the 
factor $\bar{\Delta}_i$. Then $(\bar\bbE\times \bar\bbE)\bar{M}=\bar{\Delta}$ if
and only if both $\pr_i$ are surjective. Thus
\begin{equation}
  \label{eq:54}
\dim_{\F_4}\bar{M}\geq
2\quad \text{for every}\quad \bar{M}\in \grS(\bar{\Delta}).
\end{equation}
Suppose that  $\bar{M}\in \grS(\bar{\Delta})$ from now on. 

If $\dim_{\F_4}\bar{M}=2$, then both
$\pr_i$ are $\bar\bbE$-isomorphisms, and $\bar{M}$ is the graph of the
isomorphism $\pr_2\circ \pr_1^{-1}: \bar{\Delta}_1\to
\bar{\Delta}_2$. Necessarily,  $\Delta$ is equal to either $\left[\begin{smallmatrix}
    \Z_4& \Z_4\\
    2\Z_4& 2\Z_4
  \end{smallmatrix}\right]$ or $\left[\begin{smallmatrix}
    \Z_4& \Z_4\\
    \Z_4& \Z_4
  \end{smallmatrix}\right]$. In these two cases, any isomorphism
$\bar{\Delta}_1\to \bar{\Delta}_2$ is a scalar multiplication by
$\F_4^\times$. After a suitable multiplication by an element of
$\F_4^\times\times \F_4^\times$, we may identify $\bar{M}$ with the
diagonal of $\bar{\Delta}_1\times \bar{\Delta}_2$.

Next, suppose that $\dim_{\F_4}\bar{M}=3$. Then $\ker(\pr_1)$ is 
a $1$-dimensional submodule of $\bar{M}\cap \bar{\Delta}_2$, so it must 
coincide with the unique $1$-dimensional submodule of $\bar{\Delta}_2$.
A similar result holds for $\ker(\pr_2)$.   We claim that $\bar{\Delta}_1\simeq \bar{\Delta}_2$ in this case as well. If not, without lose
of generality,  we may assume that $\bar{\Delta}_1=\left[\begin{smallmatrix}
  \F_4\\ \F_4
\end{smallmatrix}\right]^\dagger$ and $\bar{\Delta}_2=\left[\begin{smallmatrix}
  \F_4\\ \F_4
\end{smallmatrix}\right]^\ddagger$ so that $\bar{\Delta}=\bar\bbE$.  By the above discussion,
$\bar{M}\supseteq \left<
  \begin{smallmatrix}
    0 & \F_4\\ \F_4 & 0
  \end{smallmatrix}
\right>$. Since $\dim_{\F_4}\bar{M}=3$,
there exists  $u, v\in \F_4^\times$ such that
\[\bar{M}=\F_4\left<
  \begin{matrix}
    u & 0\\ 0 & v
  \end{matrix}
\right>+\F_4 \left<
  \begin{matrix}
    0 & 0\\ 1 & 0
  \end{matrix}
\right>+\F_4 \left<
  \begin{matrix}
    0 &  1\\ 0 & 0
  \end{matrix}
\right>.\]
  Indeed, both $u$ and $v$ have to be  nonzero  
 since $\pr_i$ is surjective for each $i=1, 2$.  However, $\left<
  \begin{smallmatrix}
    u & 0\\ 0 & v
  \end{smallmatrix}
\right>\in \bar\bbE^\times$, which implies that
$\bar{M}=\bar\bbE=\bar{\Delta}$. 
This  contradicts $\dim_{\F_4}\bar{M}=3$ and 
verifies our claim.  It remains to consider the cases 
 $\bar{\Delta}=\left[\begin{smallmatrix}
  \F_4\\ \F_4
\end{smallmatrix}\right]^\dagger\times \left[\begin{smallmatrix}
  \F_4\\ \F_4
\end{smallmatrix}\right]^\dagger$ or  $\bar{\Delta}=\left[\begin{smallmatrix}
  \F_4\\ \F_4
\end{smallmatrix}\right]^\ddagger\times \left[\begin{smallmatrix}
  \F_4\\ \F_4
\end{smallmatrix}\right]^\ddagger$. 
For simplicity,  we write them as 
$\Mat_2(\F_4)^\dagger$ and $\Mat_2(\F_4)^\ddagger$
respectively. Suppose that $\bar{\Delta}=\Mat_2(\F_4)^\dagger$. Then there exists $u, v\in \F_4^\times$ such
that 
\[\bar{M}=\F_4
  \begin{bmatrix}
    u & v \\ 0 & 0 
  \end{bmatrix}+   \F_4\begin{bmatrix}
    0 & 0 \\ 1 & 0 
  \end{bmatrix}+   \F_4\begin{bmatrix}
    0 & 0 \\ 0 & 1 
  \end{bmatrix}.
\]
Multiplication  by $(u^{-1}, v^{-1})\in \F_4^\times\times
\F_4^\times$ sends $\bar{M}$ to 
\[\bar{M}_0:=\left\{
    \begin{bmatrix}
      a & b \\ c & d
    \end{bmatrix}\in \Mat_2(\F_4)^\dagger\, \middle\vert\,  a=b
  \right\}.\]
Thus all 3-dimensional members of $\grS(\Mat_2(\F_4)^\dagger)$
are in the same $(\F_4^\times\times
\F_4^\times)$-orbit. 
A similar result holds for $\bar{\Delta}=\Mat_2(\F_4)^\ddagger$. 

Lastly, if $\dim_{\F_4}\bar{M}=4$, then $\bar{M}=\bar{\Delta}$. 

Combining (\ref{eq:49}), (\ref{eq:51}) and the above classifications,
we find that
\[\abs{\scrL_2(\uln)}=\sum_{\Delta}\abs{(\F_4^\times\times
    \F_4^\times)\backslash \grS(\bar{\Delta})}= 3+3+1+1=8.\]
Thus $o(3, 6)=8$ if $p=2$ according to Lemma~\ref{lem:3-6-step1}. 
\end{proof}

\begin{prop}\label{lem:o-2-2p}
  Suppose that $p\in \{2, 3\}$. Then
  \begin{equation}
    \label{eq:56}
    o(2, 2p)=
    \begin{cases}
      2 &\text{if } p=2,\\
      3&\text{if } p=3.
    \end{cases}
  \end{equation}
\end{prop}
\begin{proof}
  Let $\uln=(n_1, n_2)=(2, 2p)$ for $p\in \{2, 3\}$.  
  Then $K_\uln=\Q\times K_{n_2}$, and
  $\scrK_\uln\simeq D\times \Mat_2(K_{n_2})$ by (\ref{eq:30}) and
  (\ref{eq:29}).  Here $K_{n_2}$ is equal to   $\Q(\sqrt{-1})$ or
  $\Q(\sqrt{-3})$ according to whether $p=2$ or $p=3$. 
  Let $\grp$ be the unique ramified prime ideal
  of $A_{n_2}$ above $p$. From \cite[(5.5) and
  (5.6)]{xue-yang-yu:sp_as2}, we have 
  \begin{equation}
    \label{eq:57}
    A_\uln=\{(a,b)\in \Z\times A_{n_2}\mid a\equiv
    b\pmod{\grp}\},  
  \end{equation}
which is a suborder of index $p$ in $O_{K_\uln}=\Z\times
A_{n_2}$.

  From (\ref{eq:31}), $S(\uln,
  p)=\{p\}$. 
  Clearly, $\calO_p$ is canonically a subring of $\scrA_{n_2,
    p}=A_{n_2}\otimes_\Z\calO_p$, and $(\grp\scrA_{n_2,
    p})\cap \calO_p=p\calO_p$. It follows from (\ref{eq:57}) that 
  \begin{equation}
    \label{eq:58}
    \scrA_{\uln, p}=\{(x, y)\in \calO_p\times \scrA_{n_2, p}\mid x\equiv y
    \pmod{\grp \scrA_{n_2, p}}\}.
  \end{equation}
  For simplicity, let us put $\scrC:=\calO_p\times \scrA_{n_2, p}$, and
  $J:=\grJ(O_{K_\uln, p})=p\Z_p\times \grp$, where $\grp$ denotes 
 the maximal ideal of $A_{n_2, p}$ by an abuse of notation.  Then
  $J\scrC\subseteq \scrA_{\uln,p}$, so we further define three
  quotient rings
\begin{align}
  \label{eq:59}
\bar\calO_p&:=\calO_p/p\calO_p,\\
\bar\scrC&:=\scrC/(J\scrC)=(\calO_p/p\calO_p)\times (\scrA_{n_2,
           p}/\grp\scrA_{n_2, p})=\bar\calO_p\times \bar\calO_p, \label{eq:64}\\
  \bar\scrA_{\uln, p}&:=\scrA_{\uln,p}/(J\scrC)=\bar\calO_p,  \label{eq:66}
\end{align}
where $\bar\scrA_{\uln,p}$ embeds into $\bar\scrC$ diagonally by
(\ref{eq:58}).   
From \cite[Corollary~II.1.7]{vigneras}, $\calO_p$ contains a copy of
$\Z_{p^2}$, and there exists $\eta\in \calO_p$ such that
\begin{equation}
  \label{eq:60}
\calO_p=\Z_{p^2}+\Z_{p^2}\eta, \qquad \eta^2=p, \quad\text{and}\quad x
\eta=\eta \tilde{x}, \quad \forall x\in \Z_{p^2}.     
\end{equation}
Here $x\mapsto \tilde{x}$ denotes the unique nontrivial $\Q_p$-automorphism of
$\Q_{p^2}$. If we write $\bar\eta$ for  the   canonical image of $\eta$ in
$\bar\calO_p$, then 
\begin{equation}
  \label{eq:61}
  \bar\calO_p=\F_{p^2}+\F_{p^2}\bar\eta. 
\end{equation}
In particular, $\dim_{\F_{p^2}}\bar\calO_p=2$, and the Jacobson
radical $\grJ(\bar\calO_p)=\F_{p^2}\bar\eta$ is  the
unique $1$-dimensional submodule of  $\bar\calO_p$.

From (\ref{eq:30}) and (\ref{eq:29}), $V_{n_1}$ is a free module of
rank $1$ over $\scrK_{n_1}\simeq D$, and  $V_{n_2}=K_{n_2}^2$ is a simple
module over 
$\scrK_{n_2}\simeq \Mat_2(K_{n_2})$. Fix a suitable
identification of $\scrK_{n_2, p}=\Mat_2(K_{n_2, p})$ so that the hereditary
closure $\calH(\scrA_{n_2, p})$ is equal to $\Mat_2(A_{n_2, p})$. From
Lemma~\ref{lem:simple-module-ram}, every $\scrC$-lattice in $V\otimes
\Q_p$ is isomorphic to
\begin{equation}
  \label{eq:62}
\Delta:=\Delta_1\times \Delta_2=\calO_p\times
\begin{bmatrix}
  A_{n_2, p}\\ A_{n_2, p}
\end{bmatrix}. 
\end{equation}
Put $\bar\Delta:=\Delta/J\Delta=\bar\Delta_1\times\bar\Delta_2$, where 
\begin{equation}
  \label{eq:63}
 \bar\Delta_1=\calO_p/p\calO_p=\bar\calO_p, \quad\text{and}\quad
 \bar\Delta_2=\Delta_2/\grp\Delta_2. 
\end{equation} The same proof as (\ref{eq:51}) shows that there is a bijection
\begin{equation}
  \label{eq:65}
    \left\{\parbox{1.2in}{Isomorphism classes of $\scrA_{\uln,
        p}$-lattices  in $V\otimes \Q_p$}
  \right\}\longleftrightarrow\left\{\parbox{1.8in}{$\End_{\scrC}(\Delta)^\times$-equivalent classes of $\bar\scrA_{\uln, p}$-submodules $\bar M$ in
$\bar \Delta$ such that $\bar\scrC \bar M=\bar \Delta$}\right\}. 
\end{equation}
Clearly, 
 $\dim_{\F_p} \bar\Delta_2=2$. When regarded as an $\scrA_{n_2,
  p}$-module, $\bar\Delta_2$ is
isomorphic to the unique simple $\scrA_{n_2, p}$-module $\F_{p^2}$ by
(\ref{eq:64}) and (\ref{eq:61}). We fix such an
isomorphism and 
write
$\bar\Delta_2=\F_{p^2}$. Note that 
$\End_{\scrC}(\Delta)=\calO_p\times A_{n_2, p}$, so the action of 
$\End_{\scrC}(\Delta)^\times$   on $\bar \Delta$ factors through
$\bar\calO_p^\times\times \F_p^\times$. Recall that
$\bar\scrC=\bar\calO_p\times \bar\calO_p$ and 
$\bar\scrA_{\uln, p}=\bar\calO_p$ by (\ref{eq:64}) and (\ref{eq:66}). 

Let $\bar{M}\subseteq \bar\Delta$  be an $\bar\calO_p$-submodule, and $\pr_i:
\bar{M}\to \Delta_i$ be the canonical projections for $i=1,2$.  Then  
$(\bar\calO_p\times \bar\calO_p)\bar{M}=\bar\Delta$ if and only if
both $\pr_i$ are surjective. Suppose that this is the case. Then
necessarily \[\dim_{\F_{p^2}}\bar{M}\geq
\dim_{\F_{p^2}}\bar\Delta_1=2. \]
If $\dim_{\F_{p^2}}\bar{M}=3$, then $\bar{M}=\bar\Delta$.

Now suppose that $\dim_{\F_{p^2}}\bar{M}=2$. Then $\ker(\pr_2)$ is a
$1$-dimensional submodule of $\bar\Delta_1$, so $
\ker(\pr_2)=\F_{p^2}\bar\eta$. Therefore, there exists $a, c\in
\F_{p^2}^\times$ such that
\begin{equation}
  \label{eq:67}
  \bar{M}=\F_{p^2}(\bar\eta, 0)+\F_{p^2}(a, c)\subseteq \bar\Delta=
  (\F_{p^2}+\F_{p^2}\bar\eta)\times \F_{p^2}.
\end{equation}
Indeed,  both $a, c$ have to be nonzero since $\pr_i$ is surjective for
each $i=1, 2$.  Multiplication  by $(ca^{-1}, 1)\in
\bar\calO_p^\times\times \F_p^\times$ sends  $\bar{M}$ to the
following submodule of $\bar\Delta$:
\begin{equation}
  \label{eq:68}
  \bar{\Gamma}:=\{(a+b\bar\eta, c)\in \bar\Delta\mid a=c \}.
\end{equation}
Let $\Gamma$ be the unique $\scrA_{\uln, p}$-sublattice of $\Delta$ such
that $\Gamma \supseteq J\Delta$ and $\Gamma/J\Delta=\bar{\Gamma}$. Then every
$\scrA_{\uln, p}$-lattice in $V\otimes\Q_p$ is isomorphic to either
$\Delta$ or $\Gamma$.   Let   $\grP$ be the unique two-sided prime ideal of
$\calO_p$.
We compute
\begin{equation}
  \label{eq:69}
  \begin{split}
  \End_{\scrA_{\uln, p}}(\Gamma)&=\{(x,y)\in \calO_p\times A_{n_2, p}\mid 
  (x, y)\bar{\Gamma}\subseteq \bar{\Gamma}\}\\    
   &=\{(x,y)\in \calO_p\times A_{n_2, p}\mid x\equiv y \pmod{\grP}\}.
  \end{split}
\end{equation}
Here $\grP\cap A_{n_2, p}=\grp $ for any embedding of
$A_{n_2, p}$ into $\calO_p$, so the congruence relation does not depend
on the choice of such an embedding.

We have seen from above that there are two genera of $\scrA_\uln$-lattices
in $V$. According to  \cite[Proposition~V.3.1]{vigneras}, $h(\calO)=1$ for
$p\in \{2, 3\}$, so $\calO$ is the unique maximal order in $D$ up to
$D^\times$-conjugation. If $\wt\Delta$ is an $\scrA_\uln$-lattices
in $V$ with $\wt\Delta\otimes \Z_p=\Delta$, then
\begin{equation}
  \label{eq:70}
  \End_{\scrA_\uln}(\wt\Delta)\simeq \calO\times A_{n_2}.   
\end{equation}
Similarly, if $\wt\Gamma$ is the unique $\scrA_\uln$-sublattice of
$\wt\Delta$ such that 
$\wt\Gamma\otimes \Z_p=\Gamma$, then $  \End_{\scrA_\uln}(\wt\Gamma)$
is the unique suborder of   $\End_{\scrA_\uln}(\wt\Delta)$ satisfying 
\begin{equation}
  \label{eq:71}
  \End_{\scrA_\uln}(\wt\Gamma)\otimes \Z_\ell=
  \begin{cases}
    \End_{\scrA_\uln,p}(\Gamma) &\text{if } \ell= p,\\
      \End_{\scrA_\uln}(\wt\Delta)\otimes \Z_\ell&\text{otherwise}. 
  \end{cases}
\end{equation}
For simplicity, let us  put $\grO:=  \End_{\scrA_\uln}(\wt\Delta)$ 
and $\grR=\End_{\scrA_\uln,p}(\Gamma)$. 
From the general
strategy explained in Section~\ref{sec:gen-strategy}, we have
\begin{equation}
  \label{eq:72}
  o(2, 2p)=h(\grO)+h(\grR) \qquad \text{for } p=2,3.  
\end{equation}
Here $h(\grO)=h(\calO)h(A_{n_2})=1\cdot 1=1$. It remains to compute
$h(\grR)$.

Replacing $\wt\Delta$ by 
$g\wt\Delta$ for a suitable $g\in \scrE_\uln^\times$ if necessary, we
 assume that $\grO=\calO\times A_{n_2}$.  Let $\wh\grO$ (resp.~$\wh\grR$) be the profinite completion of $\grO$
(resp.~$\grR$). We apply
\cite[(6.3)]{xue-yang-yu:sp_as2}  to obtain
\begin{equation}
  \label{eq:73}
  h(\grR)=\abs{\grO^\times\backslash \wh
    \grO^\times/\wh\grR^\times}=\abs{\grO^\times\backslash \grO_p^\times/\grR_p^\times}.
\end{equation}
From (\ref{eq:69}), $\grP\times \grp \subseteq \grR_p$, and
\[\grO_p/(\grP\times \grp)=\F_{p^2}\times \F_p,\qquad  \grR_p/
  (\grP\times \grp)=\F_p. \]
where $\F_p$ embeds into $\F_{p^2}\times \F_p$ diagonally. It follows
that $\grO_p^\times/\grR_p^\times$ can be further simplified into
$(\F_{p^2}^\times\times \F_p^\times)/\diag(\F_p^\times)$. On the other
hand, $\grO^\times=\calO^\times\times A_{n_2}^\times$. Since $p\in \{2,
3\}$, the natural
map $\grO^\times\to \F_{p^2}^\times\times \F_p^\times$ sends
$\{\pm 1\}\times A_{n_2}^\times$ onto $\F_p^\times\times \F_p^\times$,
which contains $\diag(\F_p^\times)$. Therefore, 
\begin{equation}
  \label{eq:74}
  \begin{split}
  h(\grR)&=\abs{\grO^\times\backslash \grO_p^\times/\grR_p^\times}=\abs{(\calO^\times\times A_{n_2}^\times)\backslash (\F_{p^2}^\times\times
    \F_p^\times)/\diag(\F_p^\times)}\\&=\abs{(\calO^\times\times \{1\})\backslash(\F_{p^2}^\times\times
    \F_p^\times)/(\F_p^\times\times \F_p^\times)}=\abs{\calO^\times\backslash
    \F_{p^2}^\times/\F_p^\times}=\abs{\calO^\times\backslash
    \F_{p^2}^\times}.    
  \end{split}
\end{equation}
Here we have used freely the commutativity of $\F_{p^2}^\times\times
    \F_p^\times$. 



If $p=2$, we have already shown  in the proof of Lemma~\ref{lem:o-1-2}
that the canonical map $\calO^\times\to
(\calO/2\calO)^\times$ is surjective (see (\ref{eq:76})). Consequently,
$\calO^\times\to (\calO_p/\grP)^\times=\F_4^\times$ is surjective
as well. Thus $h(\grR)=1$ if $p=2$.  

Next, suppose that $p=3$. According to \cite[Exercise~III.5.2]{vigneras}, 
\[D=\qalg{-1}{-3}{\Q}, \quad \text{and} \quad \calO=\Z+\Z i+\Z\frac{1+j}{2}+\Z\frac{i(1+j)}{2}.\]
Here $\{1, i, j, ij\}$ is the standard $\Q$-basis of $D$. We have
\[\calO^\times=\left\{\pm 1,\quad \pm i,\quad \pm \frac{1\pm j}{2},
    \quad \pm \frac{i(1\pm j)}{2}\right\}.\]
Note that $j\in \grP$, so the image of $\calO^\times$ in
$(\calO_p/\grP)^\times=\F_9^\times$ is equal to $\{\pm \bar 1, \pm
\bar i\}$,
which has index $2$ in $\F_9^\times$. Therefore, $h(\grR)=2$ if
$p=3$. 

In conclusion, we find that $o(2, 4)=1+1=2$ if $p=2$, and $o(2,
6)=1+2=3$ if $p=3$. 
\end{proof}

\begin{rem}
Assume that $p\in \{2, 3\}$. We have shown that except for $(\uln,p)=((2,6), 3)$ (and $((1,3), 3)$ by \cite[Remark~3.2(i)]{xue-yang-yu:sp_as2}), every genus in the set $\scrL(\uln)$ of lattice classes has class number one. For the exceptional cases, $\scrL(\uln)$ consists of two genera; one genus has class number one while the other one has class number two. 
\end{rem}

\begin{proof}[Concluding the Proof of Theorem~\ref{thm:main}]
For $p=2,3,5$, the value of $o(\uln)$ is listed in the following table
(see 
\cite{xue-yang-yu:sp_as2} 
for the values not covered in the present paper):

\begin{center}
\begin{tabular}{*{13}{|>{$}c<{$}}|}
  \hline
  \uln    & 3 & 4 & 5 & 8 & 12  & (1, 2) & (2, 3) & (2, 4) & (2, 6) & (3,4) & (3,6)   \\
  \hline
  p=2 & 3 & 2 & 2 & 1 & 3 & 3 & 2 & 2 & 4 & 2 & 8   \\
  \hline
  p=3 & 2 & 3 & 2 & 4 & 3 & 3 & 1 & 4 & 3 & 2& 2 \\
  \hline
  p=5 & 3 & 1 & 1 & 4 & 4 & 4 & 2 & 0 & 6 & 0 & 8  \\
  \hline
\end{tabular}
\end{center}
Formula (\ref{eq:9}) is obtained by plugging in the above values in
\eqref{eq:1}. 
\end{proof}

\section*{Acknowledgments}
  J.~Xue is partially supported
by Natural Science Foundation of China grant \#11601395. 
C.-F.~Yu is partially supported
by the MoST grant 109-2115-M-001-002-MY3.

\bibliographystyle{hplain}
\bibliography{TeXBiB}

\def\cprime{$'$}
\begin{thebibliography}{10}

\bibitem{Bass-MathZ-1963}
Hyman Bass.
\newblock On the ubiquity of {G}orenstein rings.
\newblock {\em Math. Z.}, 82:8--28, 1963.

\bibitem{Brzezinski-1983}
J.~Brzezi\'nski.
\newblock On orders in quaternion algebras.
\newblock {\em Comm. Algebra}, 11(5):501--522, 1983.

\bibitem{Brzezinski-loc-Princ}
J.~Brzezinski.
\newblock Riemann-{R}och theorem for locally principal orders.
\newblock {\em Math. Ann.}, 276(4):529--536, 1987.

\bibitem{Brzezinski-crelle-1990}
Juliusz Brzezinski.
\newblock On automorphisms of quaternion orders.
\newblock {\em J. Reine Angew. Math.}, 403:166--186, 1990.

\bibitem{Voight-basic-orders}
Sara {Chari}, Daniel {Smertnig}, and John {Voight}.
\newblock {On basic and Bass quaternion orders}.
\newblock {\em arXiv e-prints}, page arXiv:1903.00560, March 2019,
  \href{https://arxiv.org/abs/1903.00560}{{\ttfamily arXiv:1903.00560}}.

\bibitem{curtis-reiner:1}
Charles~W. Curtis and Irving Reiner.
\newblock {\em Methods of representation theory. {V}ol. {I}}.
\newblock Wiley Classics Library. John Wiley \& Sons, Inc., New York, 1990.
\newblock With applications to finite groups and orders, Reprint of the 1981
  original, A Wiley-Interscience Publication.

\bibitem{Deuring-1941}
Max Deuring.
\newblock Die {T}ypen der {M}ultiplikatorenringe elliptischer
  {F}unktionenk\"{o}rper.
\newblock {\em Abh. Math. Sem. Hansischen Univ.}, 14:197--272, 1941.

\bibitem{Deuring1950}
Max Deuring.
\newblock Die {A}nzahl der {T}ypen von {M}aximalordnungen einer definiten
  {Q}uaternionenalgebra mit primer {G}rundzahl.
\newblock {\em Jber. Deutsch. Math. Verein.}, 54:24--41, 1950.

\bibitem{Drozd-Kirichenko-Roiter-1967}
Ju.~A. Drozd, V.~V. Kiri\v{c}enko, and A.~V. Ro\u{\i}ter.
\newblock Hereditary and {B}ass orders.
\newblock {\em Izv. Akad. Nauk SSSR Ser. Mat.}, 31:1415--1436, 1967.

\bibitem{eichler-CNF-1938}
Martin Eichler.
\newblock \"{U}ber die {I}dealklassenzahl total definiter
  {Q}uaternionenalgebren.
\newblock {\em Math. Z.}, 43(1):102--109, 1938.

\bibitem{igusa}
Jun-ichi Igusa.
\newblock Class number of a definite quaternion with prime discriminant.
\newblock {\em Proc. Nat. Acad. Sci. U.S.A.}, 44:312--314, 1958.

\bibitem{Kaplansky-quat-invertible}
Irving Kaplansky.
\newblock Submodules of quaternion algebras.
\newblock {\em Proc. London Math. Soc. (3)}, 19:219--232, 1969.

\bibitem{karemaker2020mass}
Valentijn Karemaker, Fuetaro Yobuko, and Chia-Fu Yu.
\newblock Mass formula and oort's conjecture for supersingular abelian
  threefolds, 2020,  \href{https://arxiv.org/abs/2002.10960}{{\ttfamily
  arXiv:2002.10960}}.

\bibitem{Lam-noncom-ring}
T.~Y. Lam.
\newblock {\em A first course in noncommutative rings}, volume 131 of {\em
  Graduate Texts in Mathematics}.
\newblock Springer-Verlag, New York, 1991.

\bibitem{li-xue-yu:unit-gp}
Qun {Li}, Jiangwei {Xue}, and Chia-Fu {Yu}.
\newblock {Unit groups of maximal orders in totally definite quaternion
  algebras over real quadratic fields}.
\newblock {\em ArXiv e-prints}, July 2018,
  \href{https://arxiv.org/abs/1807.04736}{{\ttfamily arXiv:1807.04736}}.
\newblock To appear in {Trans. Amer. Math. Soc.}

\bibitem{mumford:av}
David Mumford.
\newblock {\em Abelian varieties}, volume~5 of {\em Tata Institute of
  Fundamental Research Studies in Mathematics}.
\newblock Published for the Tata Institute of Fundamental Research, Bombay,
  2008.
\newblock With appendices by C. P. Ramanujam and Yuri Manin, Corrected reprint
  of the second (1974) edition.

\bibitem{peng-xue:select}
Deke Peng and Jiangwei Xue.
\newblock Optimal spinor selectivity for quaternion {Bass} orders.
\newblock {\em arXiv e-prints}, December 2020,
  \href{https://arxiv.org/abs/2012.01117}{{\ttfamily arXiv:2012.01117}}.

\bibitem{Roggenkamp-Latt-II}
Klaus~W. Roggenkamp.
\newblock {\em Lattices over orders. {II}}.
\newblock Lecture Notes in Mathematics, Vol. 142. Springer-Verlag, Berlin-New
  York, 1970.

\bibitem{Serre-lem-to-Grothen}
Jean-Pierre Serre.
\newblock Rigidit\'e du foncteur de jacobi d'\'echelon $n\geq 3$, {A}ppendix to
  {A}. {G}rothendieck, {T}echniques de construction en g\'eom\'etrie
  analytique. {X}. {C}onstruction de l'espace de {T}eichm\"uller.
\newblock {\em S\'eminaire Henri Cartan}, 13(2), 1960-1961.
\newblock talk:17.

\bibitem{Zarhin-Silverberg:Serre-lem}
A.~Silverberg and Yu.~G. Zarhin.
\newblock Variations on a theme of {M}inkowski and {S}erre.
\newblock {\em J. Pure Appl. Algebra}, 111(1-3):285--302, 1996.

\bibitem{Swan-1988}
Richard~G. Swan.
\newblock Torsion free cancellation over orders.
\newblock {\em Illinois J. Math.}, 32(3):329--360, 1988.

\bibitem{vigneras}
Marie-France Vign{\'e}ras.
\newblock {\em Arithm\'etique des alg\`ebres de quaternions}, volume 800 of
  {\em Lecture Notes in Mathematics}.
\newblock Springer, Berlin, 1980.

\bibitem{Washington-cyclotomic}
Lawrence~C. Washington.
\newblock {\em Introduction to cyclotomic fields}, volume~83 of {\em Graduate
  Texts in Mathematics}.
\newblock Springer-Verlag, New York, second edition, 1997.

\bibitem{waterhouse:thesis}
William~C. Waterhouse.
\newblock Abelian varieties over finite fields.
\newblock {\em Ann. Sci. \'Ecole Norm. Sup. (4)}, 2:521--560, 1969.

\bibitem{xue-yang-yu:num_inv}
Jiangwei Xue, Tse-Chung Yang, and Chia-Fu Yu.
\newblock Numerical invariants of totally imaginary quadratic {$\Bbb Z[\sqrt
  p]$}-orders.
\newblock {\em Taiwanese J. Math.}, 20(4):723--741, 2016.

\bibitem{xue-yang-yu:sp_as}
Jiangwei Xue, Tse-Chung Yang, and Chia-Fu Yu.
\newblock On superspecial abelian surfaces over finite fields.
\newblock {\em Doc. Math.}, 21:1607--1643, 2016.

\bibitem{xue-yang-yu:ECNF}
Jiangwei Xue, Tse-Chung Yang, and Chia-Fu Yu.
\newblock Supersingular abelian surfaces and {E}ichler's class number formula.
\newblock {\em Asian J. Math.}, 23(4):651--680, 2019.

\bibitem{xue-yang-yu:sp_as2}
Jiangwei Xue, Tse-Chung Yang, and Chia-Fu Yu.
\newblock On superspecial abelian surfaces over finite fields {II}.
\newblock {\em J. Math. Soc. Japan}, 72(1):303--331, 2020.

\bibitem{Yu-End-QM-2013}
Chia-Fu Yu.
\newblock Endomorphism algebras of {QM} abelian surfaces.
\newblock {\em J. Pure Appl. Algebra}, 217(5):907--914, 2013.

\end{thebibliography}
\end{document}